\documentclass{amsart}

\usepackage[dvipsnames]{xcolor}

\definecolor{winered}{rgb}{0.55,0,0}
\definecolor{darkblue}{rgb}{0,0.3,0.7}

\def\opac{0.35}
\PassOptionsToPackage{hyphens}{url}
\usepackage[colorlinks=true, linkcolor=darkblue , citecolor=winered, urlcolor=winered, breaklinks=true]{hyperref}

\usepackage[style=alphabetic, backend=bibtex, maxalphanames=5, maxnames=10, firstinits=true, isbn=false]{biblatex}

\DeclareFieldFormat[article]{title}{#1}
\renewbibmacro{in:}{%
  \ifentrytype{article}{}{\printtext{\bibstring{in}\intitlepunct}}}

\usepackage{amssymb}

\usepackage{lmodern,bm}    
\usepackage{tikz}
\usetikzlibrary{3d}
\usepackage{pgfplots}
\usepackage{tikz-3dplot}
\usepackage[all]{xy}
\usepackage{dsfont}
\usepackage{pdflscape}
\usepackage{bigdelim}
\usepackage{multirow}
\usepackage{wasysym}
\usepackage{enumitem}

\newtheorem{theorem}{Theorem}[section]
\newtheorem{introthm}{Theorem}
\newtheorem{conjecture}{Conjecture}
\newtheorem{introcor}[introthm]{Corollary}

\newtheorem{lemma}[theorem]{Lemma}

\newtheorem{proposition}[theorem]{Proposition}
\newtheorem{corollary}[theorem]{Corollary}

\theoremstyle{definition}

\newtheorem{definition}[theorem]{Definition}

\newtheorem{example}[theorem]{Example}
\newtheorem{remark}[theorem]{Remark}

\numberwithin{equation}{theorem}

\def\dist{{\rm dist}}
\def\diam{{\rm diam}}
\def\red{{\rm red}}
\def\Aut{{\rm Aut}}
\def\fin{{\rm fin}}
\def\Cl{{\rm Cl}}

\def\dif{{\rm d}}
\def\divisor{{\rm div}}
\def\del{\partial}
\def\delbar{\overline{\partial}}
\def\zbar{\overline{z}}
\def\xbar{\overline{x}}

\def\Link{{\rm Link}}
\def\sing{{\rm sing}}

\def\GL{{\rm GL}}

\def\id{{\rm id}}
\def\im{{\rm im}}
\def\loc{{\rm loc}}

\def\reg{{\rm reg}}

\def\sm{{\rm sm}}
\def\supp{{\rm supp}}
\def\codim{{\rm codim}}

\def\Jac{{\rm Jac}}

\def\PP{{\mathbb P}}
\def\CC{{\mathbb C}}
\def\ZZ{{\mathbb Z}}
\def\NN{{\mathbb N}}
\def\QQ{{\mathbb Q}}
\def\RR{{\mathbb R}}
\def\DD{{\mathbb D}}

\newcommand{\orbi}[1]{\mathcal{#1}}

\title[The local fundamental group of a klt singularity is finite]{The local fundamental group of a Kawamata log terminal singularity is finite}
\author{Lukas Braun}

\address{Mathematisches Institut, Albert-Ludwigs-Universit\"at Freiburg, Ernst-Zermelo-Strasse 1, 79104 Freiburg im Breisgau, Germany}
\email{lukas.braun@math.uni-freiburg.de}

\thanks{The author was partially supported by the DFG-Graduiertenkolleg GK1821 "Cohomological
Methods in Geometry" at the University of Freiburg. 
}


\begin{document}

\begin{abstract}
We prove a conjecture of Koll\'ar stating that the local fundamental group of a klt singularity $x$ is finite. In fact, we prove a stronger statement, namely that the fundamental group of the smooth locus of a neighbourhood of $x$ is finite. We call this the \emph{regional fundamental group}. As the proof goes via a local-to-global induction, we simultaneously confirm finiteness of the orbifold fundamental group of the smooth locus of a weakly Fano pair. 
\end{abstract}

\maketitle

\setcounter{tocdepth}{1}

\section*{Introduction}

We work over the field $\CC$ of complex numbers. 
A log pair is an algebraic variety $X$ together with a boundary divisor $0 \leq \Delta < 1$ of the form $\Delta =\Delta'+\Delta''$, with $0 \leq \Delta''$ and $\Delta'=\sum(1-1/m_i)\Delta_i$  is a sum of prime divisors $\Delta_i$, whose coefficients satisfy $m_i \in \ZZ_{>1}$.

A \emph{Kawamata log terminal} or \emph{klt singularity} is a point $x \in (X,\Delta)$, such that for a log resolution $f:Y \to X$, locally around $x$, the discrepancies $a_i$, namely the coefficients of the exceptional divisors $E_i$ in the formula
$$
K_Y + f^{-1}_*\Delta \sim_\QQ f^*(K_X+\Delta)+\sum a_i E_i
$$ 
satisfy $a_i > -1$. We call a log pair $(X,\Delta)$ \emph{weakly Fano}, if it has only klt singularities and $-(K_X+\Delta)$ is big and nef.  The \emph{local fundamental group} of a normal singularity $x \in X$ is 
$$
\pi_1^\loc(X,x):=\pi_1(B \setminus x) =\pi_1(\Link(x)),
$$
where $B$ is the intersection of $X$ with a small euclidean ball around $x$ and the \emph{link} $\Link(x)$ is the boundary $\del B$. It is a deformation retract of $B \setminus x$ and so $\pi_1^\loc$ is well defined.  The following conjecture is due to Koll\'ar~\cite{KollarEx, KollarSing}.

\begin{conjecture}
\label{con:local}
Let $x \in (X,\Delta)$ be a klt singularity. Then the local fundamental group $\pi_1^\loc(X,x)$ is finite.
\end{conjecture}

In the case of a weakly Fano pair $(X,\Delta)$, one can consider the smooth locus of $X$ and state the following global conjecture~\cite{Zhang, AIM}.
\begin{conjecture}
\label{con:global}
Let $(X,\Delta)$ be a weakly Fano pair. Then the fundamental group $\pi_1(X_\sm)$ of the smooth locus is finite.
\end{conjecture}

This conjecture has been proven for log del Pezzo surfaces~\cite{GurZhang1, GurZhang2, FujKobLu} and log Fano varieties of high Fano index~\cite{Zhang}.  

We prove generalized versions of both conjectures in the present paper. Firstly, for a log pair $(X,\Delta)$ with decomposition $\Delta=\Delta'+\Delta''$ as above, we can consider the \emph{complex orbifold} $\orbi{X}=(X,\Delta')$, see Section~\ref{sec:orbifold}. Then one can consider the \emph{orbifold fundamental group} of the smooth locus, denoted by $\pi_1(X_\sm,\Delta')$. This group is defined to be $\pi_1(X_\sm \setminus \supp (\Delta'))/N$, where $N$ is the normal subgroup generated by the $\gamma_i^{m_i}$, where $\gamma_i$ is a small loop around $\Delta_i$. 
Similarly to the global case, in the local setting we can consider the fundamental group  $\pi_1(B_\sm)=\pi_1(B \setminus X_\sing)$ of the smooth locus of $B$, instead of the local fundamental group. We call this group the \emph{regional fundamental group} and denote it by $\pi_1^\reg(X,x)$. It is also possible to define the orbifold fundamental group of $(B_\sm,\left. \Delta' \right|_{B_\sm})$, which we denote by $\pi_1^\reg(X,\Delta',x)$. Our two main theorems are the following.

\begin{introthm}
\label{thm:regional}
Let $x \in (X,\Delta'+\Delta'')$ be a klt singularity. Then the regional fundamental groups $\pi_1^\reg(X,x)$ and $\pi_1^\reg(X,\Delta',x)$ are finite.
\end{introthm}

\begin{introthm}
\label{thm:global}
Let $(X,\Delta'+\Delta'')$ be a weakly Fano pair. Then the orbifold fundamental group $\pi_1(X_\sm,\Delta')$ of the smooth locus is finite.
\end{introthm}

Before sketching the structure of the (simultaneous) proof of these theorems, we give a short overview of related results and state some consequences.

\subsection*{Fundamental groups of the whole space}
Fano \emph{manifolds} are known to be simply connected, and there are several proofs of this fact, relying for example on Atiyah's $L^2$-index theorem or rational connectedness. 
Generalizing the  smooth case, it was shown by Takayama~\cite{Takayama} that also weakly Fano varieties are simply connected. In fact, Takayama proves finiteness of the fundamental group of a log resolution. The corresponding local statement - simply connectedness of the preimage of a small neighbourhood of $x$ under a log resolution -  was proven by Koll\'ar~\cite{KollarShaf2} for quotient singularities and by Takayama for klt singularities~\cite{TakayamaLocalSimple}. The proof in~\cite{TakayamaLocalSimple} is  similar to that of~\cite{Takayama}, but the latter manages to avoid the $L^2$-index theorem, which turns out to be very important to us.

Simply connectedness also holds true for \emph{log canonical} pairs $(X,\Delta)$ with ample $-(K_X+\Delta)$, see~\cite{Fujino}.

\subsection*{\'Etale fundamental groups}
Conjecture~\ref{con:local} and Theorem~\ref{thm:global} have been confirmed by Xu for \emph{\'etale fundamental groups} $\hat{\pi}_1$ in~\cite[Thms.~1,~2]{Xu}, see also~\cite[Thm.~1.13]{GrebKebPet}. 
The \'etale or algebraic  fundamental group is just the profinite completion of the topological fundamental group.
Building on Xu's results, Greb, Kebekus, and Peternell showed in~\cite[Thm.~1.5]{GrebKebPet}, that a quasiprojective klt variety $X$ allows a finite quasi-\'etale cover  $Y \to X$ such that $\hat{\pi}_1(Y_\sm)$ is isomorphic to $\hat{\pi}_1(Y)$.

Analogous statements for $F$-regular singularities and strongly $F$-regular schemes in characteristic $p$ can be found in~\cite{FundGroupFreg1, FundGroupFreg2}. It is also possible to deduce the statements in characteristic zero from the ones in characteristic $p$~\cite{FundRedModP}.

Also Conjecture~\ref{con:local} was confirmed for log terminal singularities with a good torus action~\cite{LafaceLiendoMoraga}. 

In general, it is possible that $\pi_1$ is infinite while $\hat{\pi}_1$ is trivial. We give an example of such group (the Thompson group $T$) in the follow-up of this introduction.

\subsection*{Regional fundamental groups}
The regional fundamental group $\pi_1^\reg$ has - not under this name - already been considered in~\cite{Kumar, TianXu, Stibitz}. Of course, if $x$ is isolated, both local and regional fundamental group coincide.
We think that $\pi_1^\reg$ is the more natural notion for non-isolated $x$. In particular, the proof of our two main theorems would not be possible considering only $\pi_1^\loc$.

In~\cite[Thm.~2.2.6]{StibDiss}, Theorem~\ref{thm:regional} was proven for the profinite completion $\hat{\pi}_1^\reg$.
Stibitz also gave an example~\cite[Ex.~2]{Stibitz} of a non-isolated  (non-klt) singularity, where all $\hat{\pi}_1^\loc$ are finite, but $\hat{\pi}_1^\reg$ is infinite.

\subsection*{Consequences (and non-consequences) of our main theorems}
We already mentioned that building on the results of~\cite{Xu}, in~\cite[Thm.~1.5]{GrebKebPet} it was shown that a quasiprojective klt pair $(X,\Delta)$ allows a finite quasi-\'etale cover  $Y \to X$ such that $\hat{\pi}_1(Y_\sm)$ is isomorphic to $\hat{\pi}_1(Y)$. One can ask if our results imply the same statement for the topological fundamental group. Unfortunately, this is \emph{not} true. The reason is very simple: since the topological fundamental group $\pi_1(X_\sm)$ is not necessarily profinite, it can happen that there is \emph{no} finite index normal subgroup intersecting the (finite) images of the regional fundamental groups $\pi_1^\reg(X,x)$ of points $x$ in $(X,\Delta)$ nontrivially, in contrary to the case of \'etale fundamental groups in~\cite[(i)$\Rightarrow$(ii),p.~7]{Stibitz}. So even if $\pi_1(X)$ and $\hat{\pi}_1(X_\sm)$ both are trivial, it can happen that $\pi_1(X_\sm)$ is infinite.

On the other hand, by~\cite[Prop.~3.6]{TianXu}, we obtain the following direct consequence of Theorem~\ref{thm:regional}, which can be seen as an \emph{infinite} version of~\cite[Thm.~1.5]{GrebKebPet} (note that~\cite[Prop.~3.6]{TianXu} requires finiteness of the regional fundamental group).

\begin{introcor}
Let $(X,\Delta'+\Delta'')$ be a quasiprojective klt pair. Then every \'etale Galois orbifold cover of the orbifold $(X_\sm,\left. \Delta' \right|_{X_\sm})$ - that is every (possibly infinite) cover of $X_\sm$, coming from a quotient of the orbifold fundamental group $\pi_1(X_\sm,\left. \Delta' \right|_{X_\sm})$ by some normal subgroup - extends to a Galois orbifold cover of the orbifold $(X,\Delta')$. In particular, there is a (possibly infinite) Galois orbifold cover $(Y,\Delta_Y) \to (X,\Delta')$, \'etale (as orbifold cover) over $X_\sm$,  such that the orbifold fundamental groups $\pi_1(Y,\Delta_Y)$ and $\pi_1(Y_\sm,\left. \Delta_Y \right|_{Y_\sm})$  are isomorphic.
\end{introcor}  

Note that it is also possible to deduce from Theorem~\ref{thm:regional} by the  same arguments as in~\cite[Part~II,~Sec.~6.1]{GrebKebPet} an \emph{infinite} version of~\cite[Thm.~1.1]{GrebKebPet}: if $(X,\Delta)$ is a quasiprojective klt pair, then in any tower $X=X_0 \xleftarrow{\phi_1} X_1 \xleftarrow{\phi_2} X_2 \xleftarrow{\phi_3} \cdots$ of \emph{possibly infinite} quasi-\'etale covers $\phi_i$, such that $\phi_1 \circ \ldots \circ \phi_i \colon X_i \to X$ is Galois for every $i\geq 1$, all but finitely many of the $\phi_i$ are \'etale. In contrary to the finite versions from~\cite{GrebKebPet}, we have no idea if these statements could be of any use. 

We come to another  consequence.
It is known that the \emph{Cox ring} of a weakly Fano variety is finitely generated~\cite{BCHM}, and analogously, this holds for a klt quasicone~\cite{GorICR}. In particular, the divisor class group $\Cl(X)$ of any such object $X$ is finitely generated of the form $\ZZ^m \times \Cl(X)_\fin$ with $\Cl(X)_\fin$ a finite abelian group. The Cox ring is graded by $\Cl(X)$, which yields a quasi-\'etale quotient $\hat{X} \to \hat{X}/H=X$, where $\hat{X}$ is a quasiaffine variety, the so-called \emph{characteristic space}, and $H$ is a linear algebraic group of the form $(\CC^*)^m\times \Cl(X)_\fin$~\cite[Sec.~I.6.1]{coxrings}. The quotient of $\hat{X}$ by the torus $(\CC^*)^m$ yields a quasi-\'etale finite abelian Galois cover of $X$, which is universal with this property by~\cite[Prop.~2.2(iii)]{GorICR}. That means it factors through all quasi-\'etale finite abelian Galois covers of $X$. So we have the following consequence of Theorems~\ref{thm:regional} and~\ref{thm:global} - immediate from the previous discussion, which tells us that $\Cl(X)_\fin$ is the abelianization of $\pi_1(X_\sm)$.

\begin{introcor}
\label{cor:cox}
Let $X$ be a weakly Fano variety or a klt quasicone. Then the finite part of the divisor class group of $X$ is isomorphic to the first homology group of the smooth locus of $X$:
$$
\Cl(X)_\fin \cong H_1(X_\sm,\ZZ).
$$
\end{introcor}

In~\cite{GorICR}, it was shown that for weakly Fano varieties and klt quasicones an \emph{iteration of Cox rings} is finite. That is, one takes the Cox ring of $\hat{X}$, which is possible since in both cases, this space is a Gorenstein canonical quasicone~\cite[Thm.~3]{GorICR} - and iterates this procedure. After finitely many steps, one gets a  quasi-\'etale quotient $Z \to X$ by a solvable reductive group, that is a group of the form $G=(\CC^*)^m \rtimes S$. Here the \emph{iterated characteristic space} $Z$ is factorial and $S$ is a finite solvable group. The iteration of Cox rings is reflected by the derived series of $G$. The quotient of $Z$ by the normal subgroup $(\CC^*)^m$ yields a universal finite solvable cover $Z/(\CC^*)^m \to X$. Thus Theorems~\ref{thm:regional} and~\ref{thm:global} yield the following solvable version of Corollary~\ref{cor:cox}.

\begin{introcor}
\label{cor:cox2}
Let $X$ be a weakly Fano variety or a klt quasicone. Let $Z \to X$ be the iterated characteristic space with general fiber of the form $(\CC^*)^m \rtimes S$. Then the finite part $S$ is isomorphic to the solubilization of $\pi_1(X_\sm)$.
\end{introcor}

A corollary merely of the definition of the regional fundamental group is inspired by a result of Serre~\cite[Prop.~15]{Serre}, saying that any finite group is the fundamental group of a smooth projective variety. 

\begin{introcor}
Let $G$ be a finite group.
Then $G$ has a complex linear representation with no pseudoreflections.
In particular, there exists a quotient singularity $(X,x)=\CC^n/G$, such that $\pi_1^\reg(X,x)=G$.  
\end{introcor}

\begin{proof}
 Let $G$ be a finite group and $V$ a complex linear faithful representation. If $V$ has a reflection, consider the sum $V+V$. If this representation of $G$ has a reflection $g$, then consider the restricted representation $\left.V\right|_{\langle g \rangle}+\left.V\right|_{\langle g \rangle}$ of the subgroup $\langle g \rangle$, containing a pointwise fixed hyperplane $H$. Since this representation is reducible, by (the proof of)~\cite[Thm.~1]{RedFinRef}, one of the copies of $\left.V\right|_{\langle g \rangle}$ is contained in $H$. This is a contradiction, since $V$ was faithful. The quotient $(X,x):=(V+V)/G$ is thus ramified only in codimension two, so $\pi_1^\reg(X,x)=G$.
\end{proof}

\subsection*{The proof of Theorems~\ref{thm:regional} and~\ref{thm:global}}
As done by~\cite{Xu} in the case of the \'etale fundamental group, we will prove Theorems~\ref{thm:regional} and~\ref{thm:global} simultaneously by a local-to-global induction. One induction step is represented by the following two theorems.

\begin{introthm}
\label{thm:loctoglob}
Let $(Y,D'+D'')$ be an $n$-dimensional weakly Fano pair. Assume that  $n$-dimensional klt singularities $x \in (X,\Delta)$ have finite regional fundamental group. Then the orbifold fundamental group $\pi_1(Y_\sm,\left.D'\right|_{Y_\sm})$ is finite.
\end{introthm}

\begin{introthm}
\label{thm:globtoloc}
Let $x \in (X,\Delta'+\Delta'')$ be an $(n+1)$-dimensional klt singularity. 
Assume that the orbifold fundamental group $\pi_1(Y_\sm,\left.D'\right|_{Y_\sm})$ of $n$-dimensional weakly Fano pairs $(Y,D'+D'')$ is finite. Then $\pi_1^\reg(X,\Delta',x)$ is finite. 
\end{introthm}

It is clear that proving these two theorems yields a simultaneous proof of Theorems~\ref{thm:regional} and~\ref{thm:global}.

The global-to-local part Theorem~\ref{thm:globtoloc} has been proven by Tian and Xu in~\cite[Le.~3.1,3.2]{TianXu} for $\pi_1^\loc$. Then in~\cite[Le.~3.4]{TianXu}, they deduce finiteness of $\pi_1^\reg$ of a klt singularity from finiteness of $\pi_1^\loc$ for all lower dimensional klt singularities.  Unfortunately, there is a small gap in the proof, when the Seifert-van Kampen theorem is applied to certain tubular neighbourhoods of a Whitney stratification. A careful analysis is taken out in Section~\ref{sec:WorkTianXu} of the present paper. In fact, it turns out that this task is equally hard as trying to prove Theorem~\ref{thm:loctoglob} with the same methods  and assuming only finiteness of $\pi_1^\loc$ instead of $\pi_1^\reg$.

So in order for the induction to work, we really need the $\pi_1^\reg$-version of Theorem~\ref{thm:globtoloc}. When we realized that we cannot use~\cite[Le.~3.4]{TianXu}, Tian and Xu proposed to us to modify~\cite[Le.~3.1]{TianXu} for a direct proof avoiding their Lemma 3.4. After analyzing Lemma 3.1 in Section~\ref{sec:WorkTianXu}, we carry out this modification in Section~\ref{sec:Le31mod} and thus are able to prove Theorem~\ref{thm:globtoloc} in full generality.

The main part of the present paper is the proof of Theorem~\ref{thm:loctoglob}. So we have to prove finiteness of an orbifold fundamental group $\pi_1(Y_\sm,\left.D'\right|_{Y_\sm})$. In contrast to the proofs of simply connectedness of Fano manifolds using Atiyah's $L^2$-index theorem, we encounter two main difficulties. Firstly, $Y_\sm$ is not compact. Secondly, the orbifold version of the $L^2$-index theorem is more subtle, since for a universal orbifold cover $\widetilde{\orbi{X}} \to \orbi{X}$, the $L^2$-index on $\widetilde{\orbi{X}}$ is \emph{not} equal to  the Euler characteristic on $\orbi{X}$, as there  are contributions from orbifold points, see~\cite[Sec.~4.1]{TianXu}. The problems can be seen in Tian and Xu's proof of Theorem~\ref{thm:loctoglob} in the special case of $3$-dimensional Fano varieties with canonical singularities~\cite[Thm.~4.1]{TianXu}. Thus we are mildly sceptical about the possibility of proving Theorem~\ref{thm:loctoglob} in full generality using the orbifold $L^2$-index theorem.

The solution is the following. As we already mentioned, the proof of simply connectedness of weakly Fano varieties $X$ of Takayama~\cite{Takayama} manages to avoid the $L^2$-index theorem and instead relies on the so called \emph{$\Gamma$-reduction} or \emph{Shafarevich map}, independently constructed by Campana and Koll\'ar in~\cite{CampGamma1} and~\cite{KollarShaf2} for compact K\"ahler manifolds and normal proper varieties respectively. Roughly said, it parameterizes maximal subvarieties of $X$ with finite fundamental group. Takayama uses it to construct an $L^2$-section of a certain line bundle on the universal cover of $X$. By the work of Gromov~\cite{GromovKaehler}, the existence of such a section  means that  if $\pi_1(X)$ is infinite, there are many sections of the corresponding line bundle on $X$.

Fortunately, the \emph{$\Gamma$-reduction} is also available for  orbifolds  due to Claudon~\cite{claudon}. But then we still have the problem that $Y_\sm$ is not compact. This is where the hypothesis of Theorem~\ref{thm:loctoglob} comes into play (\emph{and} thus the very reason why we cannot directly prove Theorem~\ref{thm:global} but have to carry out the induction). Consider a log resolution $f:X \to Y$ of the $n$-dimensional weakly Fano pair $(Y,D'+D'')$ with exceptional prime divisors $E_i$. Then a very small loop $\gamma_i$ around a general point $e_i$ of $E_i$ can be pushed forward to $Y_\sm$ and there it lies in the smooth locus of a very small neighbourhood of the image of $e_i$, which is a klt singularity. Thus by the hypothesis saying that the regional fundamental groups of klt singularities of dimension $n$ are finite, we know that $\gamma_i$ is of finite order $m_i$ in $f^{-1}(Y_\sm  \setminus \supp(D')) = X \setminus ( \bigcup_i E_i \cup \supp(f_*^{-1}D'))$. So the normal subgroup of $\pi_1(f^{-1}(Y_\sm  \setminus \supp(D')))$ generated by all $\gamma_i^{m_i}$ is trivial. Thus  $\pi_1(Y_\sm,D')$ is isomorphic to the orbifold fundamental group of the smooth compact orbifold $(X,f_*^{-1}D'+\sum (1-1/m_i)E_i)$. 

Then the remaining task in order to prove finiteness of the latter is to transfer the techniques of~\cite{Takayama} to the orbifold setting, which is done in Part~\ref{part:loctoglob} of the present paper.

\subsection*{Possible alternative ways of proof}

We consider two alternative approaches to prove Theorems~\ref{thm:regional} and~\ref{thm:global}.

As we mentioned before, simply connectedness of Fano manifolds can be proven by showing that they are rationally connected, from which follows that their fundamental group is finite. The notion of rational connectedness can also be formulated for orbifolds, and also here, from rational connectedness of a smooth orbifold $\orbi{X}=(X, \Delta)$  (in the sense of Campana) follows finiteness of the orbifold fundamental group of $\orbi{X}$~\cite[Cor.~12.25]{CampSecOrbi}. So rational connectedness of the orbifold $(X,f_*^{-1}D'+\sum (1-1/m_i)E_i)$ supported on a log resolution of a weakly Fano pair $(Y,D)$ would  yield an alternative proof of Theorem~\ref{thm:loctoglob}. But the definition of rational connectedness for orbifolds is subtle~\cite[D\'ef.~6.11,~Rem.~6.12]{CampSecOrbi} and we have no idea how to prove it for the orbifold $(X,f_*^{-1}D'+\sum (1-1/m_i)E_i)$.

A different approach - which would yield a direct induction-free proof of Theorem~\ref{thm:global} in any dimension - is the following. In Proposition~\ref{prop:fundfin}, we prove finiteness of the orbifold fundamental group of $(X,f_*^{-1}D'+\sum (1-1/m_i)E_i)$ \emph{for any choice of $m_i$}, supported on a log resolution $X$ of a weakly Fano pair $(Y,D'+D'')$. Instead of arguing with the induction hypothesis of finiteness of the regional fundamental group of klt singularities, one also could try to show that if $\pi_1(X \setminus \supp(f_*^{-1}D' + \sum E_i))/\langle\langle \gamma_1^{m_1},\ldots,\gamma_1^{m_1}\rangle\rangle$ is finite for every choice of $m_i$, then $\pi_1(X \setminus \supp(f_*^{-1}D' + \sum E_i))$ is already finite. It is known that there are finitely presented infinite groups with trivial profinite completion, but our situation is slightly different. 

By passing to some ramified finite cover of $(Y,D''+D'')$ (which is still weakly Fano), we can assume that $\hat{\pi}_1(Y_\sm,\left. D'\right|_{Y_\sm})$ is trivial, which means that $\hat{\pi}_1(Y_\sm,\left. D'\right|_{Y_\sm})$ has no proper normal subgroups of finite index. So the normal subgroup $\langle\langle \gamma_1^{m_1},\ldots,\gamma_1^{m_1}\rangle\rangle$ is the whole group $\pi_1(X \setminus \supp(f_*^{-1}D' + \sum E_i))$ for any choice of $m_i$. This seems to be a strong property and one could ask if infinite finitely presented groups of this kind even exist. 

But they do. Mark Sapir sent us an example: the Thompson group $T$, which is simple, finitely presented, and infinite. It is generated by three elements, and two of them have infinite order, so all normal subgroups generated by any choice of powers of these elements are the whole group $T$. 

We want to remark that  $\pi_1(X \setminus \supp(f_*^{-1}D' + \sum E_i))$ is a so-called \emph{quasiprojective group}, that is the fundamental group of a smooth quasiprojective variety. These groups satisfy some strong properties, see e.g.~\cite[Sec.~1.5]{JumpLoci}. We do not know if it is possible to show that the negation of the above property is among them.

\subsection*{Structure of the paper}
In Part~\ref{part:loctoglob} of the paper, we prove Theorem~\ref{thm:loctoglob}. While the proof itself happens in Section~\ref{sec:loctoglob}, in Sections~\ref{sec:orbifold} to~\ref{sec:maxsubsp}, we review the definitions of complex orbifolds and basic related notions - e.g. of orbibundles, orbisheaves, and orbimetrics - but transfer also more sophisticated concepts for complex manifolds to the orbifold case. 

In Part~\ref{part:globtoloc}, we prove Theorem~\ref{thm:globtoloc}. After shortly recalling the notion of Whitney stratifications in Section~\ref{sec:Whit}, we carefully analyze Lemmata 3.1 and 3.4 of~\cite{TianXu} in Section~\ref{sec:WorkTianXu}. In the last Section~\ref{sec:Le31mod}, we prove  Theorem~\ref{thm:globtoloc} by modifying Lemma 3.1 appropriately.

\subsection*{Acknowledgements}
The author wishes to thank Stefan Kebekus and Joaqu\'in Moraga for several discussions related to the topics of this paper. He is grateful to Chenyang Xu and Zhiyu Tian for suggesting the modification of~\cite[Le.~3.1]{TianXu}.

Thanks go also to Mark Sapir for providing the example of the Thompson group $T$, to Andreas Demleitner for a discussion on the same topic, and to Mirko Mauri for pointing out a mistake in the definition of $E_0$ in the proof of Theorem~\ref{thm:globtoloc}.

\tableofcontents

\part{Local to global}
\label{part:loctoglob}

\section{Complex orbifolds and orbimaps}
\label{sec:orbifold}

The definition of an \emph{orbifold} - under the name of \emph{$V$-manifold} - goes back to Satake~\cite{Satake} in the real and Baily~\cite{Baily} in the complex case. The notion was then rediscovered by Thurston~\cite{Thurston}, who finally gave it the name \emph{orbifold}. Complex orbifolds are \emph{locally} - but not necessarily globally - quotients of smooth complex manifolds, which makes them complex analytic spaces with an additional local quotient structure. We will use the following definition, see e.g.~\cite[Sec.~2.1]{Comar}. 

\begin{definition}	
\label{def:orbifold}
Let $X$ be a complex analytic space of dimension $n$. An \emph{orbifold chart} on $X$ is a tuple $(U',G,\varphi,U)$, where $U' \subseteq \CC^n$ is a connected open complex analytic subspace, $G$ is a finite subgroup of the automorphism group of $U'$, and $\varphi\colon U' \to U \subseteq X$ is a proper and finite holomorphic map to the open subspace $U \subseteq X$, such that $\varphi \circ g =\varphi$ for every $g \in G$. We require the induced map $U'/G \to U$ to be a homeomorphism. 

An \emph{injection} betweeen two orbifold charts $(U',G,\varphi,U)$ and $(V',H,\psi,V)$ is a holomorphic embedding $\lambda:U' \to V'$, such that $\psi \circ \lambda = \varphi$.

An \emph{orbifold atlas} on $X$ is a family $\orbi{U}=\{(U_i',G_i,\varphi_i,U_i)\}$, such that $X=\bigcup_{i} U_i$, and, for two charts $(U_i,G_i,\varphi_i,U_i)$ and $(U_j,G_j,\varphi_j,U_j)$, and any $x \in U_i \cap U_j$, there is a third chart $(U_k,G_k,\varphi_k,U_k)$, such that $x \in U_k \subseteq U_i \cap U_j$, and there are injections $\lambda_{ik}:U_k'\to U_i'$ and $\lambda_{jk}:U_k'\to U_j'$.

An atlas $\orbi{U}$ is a \emph{refinement} of another atlas $\orbi{V}$, if for every chart $V'$ of $\orbi{V}$, there is an injection $U' \to V'$ from a chart from $\orbi{U}$. An atlas $\orbi{U}$ is \emph{maximal}, if it has no nontrivial refinement. 

Let $\orbi{U}$ be a maximal orbifold atlas on $X$. Then we call the pair $\orbi{X}:=(X,\orbi{U})$ a (complex) \emph{orbifold}.
\end{definition}

We sometimes will call $U' \to U$ a \emph{local uniformization} and $G$ the \emph{local uniformizing group}. By the slice theorem, there is always an atlas consisting of \emph{linear charts} $(\CC^n, G, \varphi,U)$, such that $G$ is a subgroup of the unitary group $U(n)$~\cite[Rem.~(5)]{MoerPronk}.

The actions of the local uniformizing groups $G \subset \Aut(U')$ and the injections $\lambda:U'\to V'$ behave well with respect to each other. 
Consider for example a chart $(U',G,\varphi,U)$ and an element $g \in G$, then since $\varphi \circ g =\varphi$ holds, $G : U' \to U'$ is an injection. 
Moreover, the following holds~\cite[Rem.~(3),~Prop.~A.1]{MoerPronk}.

\begin{lemma}
\label{le:injgroup}
Let  $\orbi{X}:=(X,\orbi{U})$ be a complex orbifold and $(U',G,\varphi,U)$, $(V',H,\psi,V)$ two orbifold charts on $\orbi{X}$. Let $\lambda, \mu \colon U' \to V'$ be two injections between the charts. Then there is a unique $h \in H$, such that $h \circ \lambda=\mu$. 

In particular, for $g \in G$ the composition $\mu:=\lambda \circ g$ defines an injection $U' \to V'$. The unique $h \in H$ with $\lambda \circ g = h \circ \lambda$ is denoted by $\lambda(g)$. The induced map $\lambda:G \to H$ is an injective group homomorphism.
\end{lemma}

The following is a direct consequence of Lemma~\ref{le:injgroup}, which we haven't found in the literature.

\begin{corollary}
Let  $\orbi{X}:=(X,\orbi{U})$ be a complex orbifold and $(U',G,\varphi,U)$ and orbifold chart on $\orbi{X}$. Let $\lambda, \mu \colon U' \to U'$ an injection from $(U',G,\varphi,U)$ to itself. Then there is a $g \in G$, such that $\lambda=g$.
\end{corollary}

Let $(U',G,\varphi,U)$ be an orbifold chart around $x \in U \subseteq X$, and $p \in \varphi^{-1}(x)$. Up to conjugacy, the isotropy subgroup $G_p$ is determined by $x$. Moreover, according to Lemma~\ref{le:injgroup}, if $(V',H,\psi,V)$ is another chart around $x$ and $q \in \psi^{-1}(x)$, then $G_p \cong H_q$, so the following is well defined up to isomorphy~\cite[Def.~4.1.2]{sasakian}. 

\begin{definition}
Let $\orbi{X}=(X,\orbi{U})$ be an orbifold and $x \in X$. For an orbifold chart $(U',G,\varphi,U)$ around $x$ and $p \in \varphi^{-1}(x)$, we call $G_x:=G_p$ the \emph{isotropy group of $x$}.
We call those $x \in X$ with $G_x=\{ e_G\}$ \emph{orbifold regular points}, and all points with $G_x\neq \{ e_G\}$ \emph{orbifold singular points}. 
\end{definition}

Note that the singular points (in the usual sense) of the complex analytic space $X$ are a subset of the orbifold singular points of $\orbi{X}=(X,\orbi{U})$. In particular, an orbifold singular point $x$ is a smooth point of $X$ if and only if $G_x$ is a \emph{reflection group} (for some and in consequence for all local uniformizations around $x$). A direct consequence of the well-definedness of the isotropy group of points of $\orbi{X}$ is the following stricter version of the already mentioned~\cite[Rem.~(5)]{MoerPronk}, which again we haven't found in the literature.

\begin{lemma}
Let  $\orbi{X}:=(X,\orbi{U})$ be a complex orbifold and $x \in X$ with isotropy group $G_x$. Then there is an orbifold chart $(\CC^n,G_x,\phi,U)$ around $x$, such that $\phi^{-1}(x)=0 \in \CC^n$ and $G_x$ acts as a subgroup of $U(n)$.
\end{lemma}

Campana in~\cite{CampFirstOrbi} introduced another notion of orbifold for pairs $(X,\Delta)$, where $\Delta$ is a certain divisor on a complex analytic space $X$. We will see that under certain conditions - which we will encounter in our setting -, his notion is equivalent to that of a complex orbifold we gave in Definition~\ref{def:orbifold}. In order to distinguish between the two notions, we will call such pairs $(X,\Delta)$ \emph{geometric orbifolds}, following~\cite{CampSchier}.

\begin{definition}
A \textit{geometric orbifold} is a pair $(X,\Delta)$, where $X$ is a complex analytic space and $\Delta$  a divisor of the form
$$
\sum_{i \in I} \left(1-\frac{1}{m_i}\right) \Delta_i,
$$
where we assume that the $m_i$ are integers greater than zero and the $\delta_i$ are prime divisors.
\end{definition}

\begin{definition}
We say that the geometric orbifold $(X,\Delta)$ is \textit{smooth}, if $X$ is a smooth complex manifold and $\supp(\Delta)$ is a simple normal crossing divisor.
\end{definition}

\begin{remark}
\label{rem:geomorb}
A smooth geometric $n$-dimensional orbifold $(X,\Delta)$ always can be represented by a complex  orbifold in the sense of Definition~\ref{def:orbifold}. Consider a local chart $\CC^n  \to V \subset X$ of $X$ as an analytic space. Then after suitable adjustment, in this chart, $\Delta$ is given by
$$
\prod_{i=1}^{k} x_i^{1-1/m_i}.
$$
So we have a \emph{local uniformization}
\begin{align*}
\CC^n &\to \CC^n \\
(x_1,\ldots,x_n) &\mapsto (x_1^{m_i},\ldots,x_k^{m_k},x_{k+1},\ldots,x_n),
\end{align*}
which is nothing but the quotient of the action of $\ZZ/m_1\ZZ \times \ldots \times \ZZ/m_k\ZZ$ acting diagonally on $\CC^n$ by roots of unity. If a local analytic chart of $X$ does not intersect $\Delta$, we can take it as orbifold chart. The compatibility of these charts is straightforward. We call this the \emph{canonical orbifold structure} of a smooth geometric orbifold. 
\end{remark}

The local uniformizing subgroups of the canonical orbifold structure of a smooth geometric orbifold are reflection groups. This can be seen from the fact that the analytic space $X$ is smooth or directly from the explicit orbifold charts in Remark~\ref{rem:geomorb}. The analogy between geometric orbifolds $(X,\Delta)$ and complex orbifolds $\orbi{X}$ actually goes further~\cite[Sec.~2]{BoyGalKoll}, but we are only interested in the particular case of \emph{smooth geometric orbifolds} here.

We close this section with the definition of \emph{orbimaps}. The original definitions from~\cite{Satake, Baily} do not in general induce morphisms of \emph{orbibundles} and \emph{orbisheaves} - which we will define in Sections~\ref{sec:orbibundles} and~\ref{sec:orbisheaves} respectively. This has been realized in~\cite{MoerPronk} and additional compatibility criteria have been introduced to remedy this problem. This led to the  equivalent notions of 'strong'~\cite{MoerPronk} and 'good'~\cite{ChenRuan} orbifold maps. Since we will work with orbibundles and -sheaves, for us the definition of a \emph{holomorphic orbimap} includes the additional compatibility criteria. That is to say, our maps are always 'strong'/'good', compare~\cite[Def.~4.1.8]{sasakian}.

\begin{definition}
\label{def:orbimap}
Let $\orbi{X}=(X,\orbi{U})$ and $\orbi{Y}=(Y,\orbi{V})$ be two complex orbifolds. A map $f:X \to Y$ is called a \emph{holomorphic orbimap} if the following hold:
\begin{enumerate}
\item For any $x \in X$, there are  orbifold charts $(U_i',G_i,\varphi_i,U_i)$ of $\orbi{X}$ around $x$ and $(V_i',H_i,\psi_i,V_i)$ of $\orbi{Y}$ around $f(x)$, such that 
\begin{enumerate}
\item $f(U) \subseteq V$ and
\item there is a holomorphic map $f'_i:U_i' \to V_i'$ satisfying $\psi \circ f'_i = f \circ \varphi$.
\end{enumerate} 
\item For any pair of charts $(U_i',G_i,\varphi_i,U_i)$ and $(U_j',G_j,\varphi_j,U_j)$  on $\orbi{X}$, any corresponding pair $(V_i',H_i,\psi_i,V_i)$ and $(V_j',H_j,\psi_j,V_j)$ of charts on $\orbi{Y}$ in the sense of item (1), and any injection $\lambda_{ji}\colon U_i' \to U_j'$, there is an injection $\mu_{ji}\colon V_i' \to V_j'$, such that 
\begin{enumerate}
\item $f'_i \circ \lambda_{ji} = \mu_{ji} \circ f'_j$ and
\item if $(U'_k,G_k,\varphi_k,U_k)$ is another chart on $\orbi{X}$ with an injection $\lambda_{ki}=\lambda_{kj} \circ \lambda_{ji}\colon U_i' \to U_k'$,  and $(V'_k,H_k,\psi_k,V_k)$ the corresponding chart on $\orbi{Y}$, then $\mu_{kj} \circ \mu_{ji} = \mu_{ki}$.
\end{enumerate}
\end{enumerate}
\end{definition}

\begin{remark}
In the setting of Definition~\ref{def:orbimap}, consider an injection $\lambda_{ji}\colon U_i' \to U_j'$ and two different injections $\mu_{ji}\colon V_i' \to V_j'$ and $\mu^*_{ji}\colon V_i' \to V_j'$ both meeting the requirements of Item (2), (a).  
Then according to Lemma~\ref{le:injgroup}, there is a \emph{unique} $h \in H_j$, such that $\mu_{ji}= h \circ \mu_{ji}^*$. So the $\mu_{ji}$ are determined only up to multiplication with elements of $H_j$. 

Now let $i=j$ and consider an injection $\lambda_{ji}=g \colon U_i' \to U_i'$ given by an element $g \in G_i$. \emph{In contrary} to the second statement of Lemma~\ref{le:injgroup}, now there is not necessarily a \emph{unique} $h \in H_i$, such that $f'_i \circ g = h \circ f'_i$.
But if we fix an assignment $\lambda_{ji} \mapsto \mu_{ji}$ between injections on $\orbi{X}$ and $\orbi{Y}$ fulfilling the requirements of Definition~\ref{def:orbimap}, then for each $i$, we get \emph{group homomorphisms} $G_i \to H_i$ for all $i$~\cite[Sec.~4.4]{OrbiGromovWitten}.
\end{remark}

A system of charts on $\orbi{X}$ and $\orbi{Y}$ fulfilling Item (1) of Definition~\ref{def:orbimap} together with an assignment $\lambda_{ji} \mapsto \mu_{ji}$ between injections of such charts is called a \emph{compatible system} in~\cite[Sec.~4.4]{OrbiGromovWitten}. If a  map between orbifolds allows a compatible system, it is called 'good'~\cite[Def.~4.4.1]{OrbiGromovWitten}. The problem is that one map may allow different compatible systems, as the following easy example shows~\cite[Ex.~4.4.2b]{OrbiGromovWitten}.

\begin{example}
\label{ex:diffcompsys}
Consider $\orbi{X}=(\CC,\orbi{U})$ with $\orbi{U}=\{(\CC,\ZZ/2\ZZ,\{x \mapsto x^2\},\CC)\}$ and $\orbi{Y}=(\CC^2,\orbi{V})$ with $\orbi{V}=\{(\CC^2,(\ZZ/2\ZZ)^2,\{(x,y) \mapsto (x^2,y^2)\},\CC^2)\}$. Both $\orbi{X}$ are smooth orbifolds. Consider the map $f: x \mapsto (x,0)$ between the underlying spaces. Then it is clear that the two possible lifts of $f$ in the orbifold charts are $x \mapsto (x,0)$ and $x \mapsto (-x,0)$. 

But there are also essentially different compatible systems. For $g=1 \in \ZZ/2\ZZ$, acting on $\CC$ by $x \mapsto -x$, it is possible to choose $h=(1,0) \in (\ZZ/2\ZZ)^2$, acting by $(x,y) \mapsto (-x,y)$, or $h'=(1,1) \in (\ZZ/2\ZZ)^2$, acting by $(x,y) \mapsto (-x,-y)$. Both choices meet the requirements of Definition~\ref{def:orbimap}. 
\end{example}

As~\cite[Le.~4.4.3]{OrbiGromovWitten} states, for any compatible system, there is a unique pullback of orbibundles. But for different compatible systems as in Example~\ref{ex:diffcompsys}, these pullbacks may differ. 

Nevertheless, the only holomorphic orbimaps we encounter are \emph{orbifold (universal) covers}. These always have a unique compatible system, since they are locally trivial in the orbifold sense.

\section{Orbibundles}
\label{sec:orbibundles}

In this section, we define orbifold vector bundles or \emph{orbibundles} as a reasonable generalization of (complex) vector bundles over (complex) manifolds. The probably most important notion is that of the \emph{orbifold tangent space} $T\orbi{X}$ and related constructions.

\begin{definition}
\label{def:orbibundle}
Let $\orbi{X}=(X,\orbi{U})$ be a complex orbifold. An \emph{orbifold vector bundle} or \emph{orbibundle} of rank $k$ over $\orbi{X}$ is a collection of vector bundles $\pi'_i \colon E'_{i} \to U_i'$ with fiber $\CC^k$ for each orbifold chart $(U_i',G_i,\varphi_i,U_i)$ of $\orbi{X}$, together with an action of $G_i$ on $E'_{i}$ by (ordinary) vector bundle maps, such that:
\begin{enumerate}
\item Each $\pi'_i$ is $G_i$-invariant, so that the following diagram is commutative for any $g \in G_i$:
$$
 \xymatrix{ 
E'_{i} \ar[r]^{g} \ar[d]^{\pi'_i}
&
E'_i \ar[d]^{\pi'_i} \\ 
U_i' \ar[r]^{g} 
&
U_i'.
}
$$
\item For any injection $\lambda_{ji}:U_i' \to U_j'$ of charts on $\orbi{X}$, there is a bundle isomorphism 
$\lambda_{ji}':E_i' \to \left.E_j'\right|_{\im(\lambda_{ji})}$, such that $\lambda_{ji}' \circ g = \lambda_{ji}(g) \circ \lambda_{ji}'$, where by $\lambda_{ji}: G_i \to G_j$ we denote the injective group homomorphism from Lemma~\ref{le:injgroup} as well.
\item For two injections $\lambda_{ji}:U_i' \to U_j'$ and $\lambda_{kj}:U_j' \to U_k'$, we have $(\lambda_{kj} \circ \lambda_{ji})'= \lambda_{kj}' \circ \lambda_{ji}'$.
\end{enumerate}

\end{definition}

\begin{remark}
\label{rem:loctriv}
The total space $E$ of an orbibundle is obtained from the local bundles $E'_i$ in the following way~\cite[Sec.~2.2]{Comar}. Choosing small enough orbifold charts on $\orbi{X}$, we can assume that $E_i' \cong U_i' \times \CC^k$ and the action of $G_i$ on $U_i' \times \CC^k$ is diagonal and acting as a subgroup of $\GL(k)$ on the second factor. Then since $\pi_i'$ is equivariant, setting $E_i:=E_i'/G_i$, we have a unique 'projection' $\pi_i$, so that the following diagram commutes:
$$
 \xymatrix{ 
E'_{i} \ar[r] \ar[d]^{\pi'_i}
&
E_i \ar[d]^{\pi_i} \\ 
U_i' \ar[r]^{\varphi_i} 
&
U_i.
}
$$
Now we can glue the sets $E_i$ in the following way, stemming from the gluing condition on $\orbi{X}$:
let $x \in U_i \cap U_j \neq \emptyset$. Then according to Definition~\ref{def:orbifold} there is a chart $x \in U_k$ with injections $\lambda_{ik}:U'_k \to U_i'$ and $\lambda_{jk}:U'_k \to U_j'$, which by Definition~\ref{def:orbibundle} (2) induce bundle isomorphisms $\lambda_{jk}':E_k' \to \left.E_j'\right|_{\im(\lambda_{jk})}$ and $\lambda_{ik}':E_k' \to \left.E_i'\right|_{\im(\lambda_{ik})}$. Gluing $E_i$ and $E_j$ acccording to this data results in an orbifold $\orbi{E}$ with underlying space $E$ and an orbimap $\pi \colon \orbi{E} \to \orbi{X}$, which is locally given by the equivariant projections $\pi_i' \colon E_i' \to U_i'$~\cite[Sec.~2.2]{Comar}.
\end{remark}

\begin{example}
\label{ex:trivline}
Probably the easiest but still important example of an orbibundle is the \emph{trivial line bundle}, given by trivial line bundles $E_i':=U_i' \times \CC$ on each chart $U_i'$ \emph{together} with a \emph{trivial} action of $G_i$ on the second factor. Then clearly $E_i \cong U_i \times \CC$ and the total space $\orbi{E}$ is just $\orbi{X} \times \CC$. 
\end{example}

\begin{example}
Another very important example is that of the \emph{tangent orbibundle} $T\orbi{X}$. It can be constructed in the following natural way~\cite[Ex.~4.2.10]{sasakian}. On a chart $U_i'$, take the tangent bundle $TU_i'\cong U_i' \times \CC^n$ and for any injection $\lambda_{ji}:U_i' \to U_j'$ of charts on $\orbi{X}$, the bundle isomorphism 
$\lambda_{ji}':E_i' \to \left.E_j'\right|_{\im(\lambda_{ji})}$ is given by $\lambda_{ji}$ on the first factor and the \emph{Jacobian} $\Jac (\lambda_{ji})$ on the second one. This construction obviously generalizes to the \emph{cotangent orbibundle} $T^*\orbi{X}$, (symmetric, antisymmetric) tensor orbibundles et cetera~\cite{OrbifoldSurvey}.
\end{example}

\begin{remark}
Locally  around $x \in X$, the fiber $\pi^{-1}(x) \subseteq T\orbi{X}$ is not isomorphic to $\CC^n$, but is holomorphic to a small neighbourhood of $x \in X$, because  in a local chart, the actions of $g \in G_i$ on $U_i'$ and of $\Jac(g)$ on $T_{\phi^{-1}(x)}U_i'$ are essentially the same.

On the other hand, even if $(X,\Delta)$ is a smooth geometric orbifold with canonical orbifold structure $\orbi{X}$, the underlying space of $T\orbi{X}$  is \emph{not necessarily} the ordinary tangent space $TX$. 
\end{remark}

Now having defined orbibundles, we have to ask ourselves what is a reasonable definition of (holomorphic) \emph{sections} of these. Obviously for an orbibundle $\pi: \orbi{E} \to \orbi{X}$, a section of $\orbi{E}$ should be a holomorphic orbimap $s:\orbi{X} \to \orbi{E}$ satisfying $\pi \circ s = \id_X$. But what does this mean on a local chart $\pi'_i\colon E'_i \to U'_i$? As we have an action of $G_i$ on $U_i'$ and $E_i'$, $s$ locally corresponds to an \emph{equivariant}  holomorphic section $s_i:U_i' \to E_i'$, meaning $g \circ s_i = s_i \circ g$ for any $g \in G_i$. Of course the local sections must be compatible with injections as well, so that we arrive at the following definition~\cite[Def.~4.2.9]{sasakian}.

\begin{definition}
\label{def:orbsecdiff}
Let $\pi: \orbi{E} \to \orbi{X}$ be an orbibundle. Then a holomorphic \emph{section} of $\orbi{E}$ is given by any of the two equivalent definitions:
\begin{enumerate}
\item $s:\orbi{X} \to \orbi{E}$ is a holomorphic orbimap satisfying $\pi \circ s = \id_\orbi{X}$.
\item A collection of holomorphic sections $s_i:U_i' \to E_i'$ of the local bundles over charts of $\orbi{X}$, such that for any injection $\lambda_{ji}\colon U_i \to U_j$, the following diagram commutes:
$$
 \xymatrix@C=3em{ 
E'_{i} \ar[r]^{\lambda_{ji}'} 
&
\left.E'_j\right|_{\im(\lambda_{ji})}  \\ 
U_i' \ar[r]^{\lambda_{ji}} \ar[u]_{s_i}
&
\im(\lambda_{ji}) \ar[u]_{\left.s_j\right|_{\im(\lambda_{ji})}}.
}
$$
\end{enumerate}
\end{definition}

\begin{remark}
\emph{Equivariance} of the local sections $s_i$ obviously is the right requirement, otherwise they would not glue to a global section $s:\orbi{X} \to \orbi{E}$. When (locally) the action of $G_i$ on the fiber is \emph{trivial}, then of course \emph{equivariance} means nothing else than \emph{invariance}, as it is the case for the trivial line bundle from Example~\ref{ex:trivline}. Sections of this bundle clearly are in a one-to-one-correspondence with holomorphic orbimaps from $\orbi{X}$ to $\CC$ endowed with the trivial orbifold structure. So they are a good candidate for a \emph{structure orbisheaf} on $\orbi{X}$, see Section~\ref{sec:orbisheaves}.
\end{remark}

In order to get \emph{coherent sheaves} on the underlying space $X$, nonetheless, we have to deal with \emph{invariant} sections of line bundles $E_i' \to U_i'$ or sheaves $\mathcal{F}'$ on the local unifomizations $U_i'$. 

The other way round works quite as well. If the underlying space $X$ of a complex orbifold is a manifold, then line bundles and Weil divisors coincide and we can pull them back to the local uniformizations, so they give orbibundles on $\orbi{X}$. Now for example, we can ask ourselves which divisor on $X$ gives the canonical orbibundle $K_{\orbi{X}}$.

\begin{example}
\label{ex:kandiv}
To answer this question, we just have to pull back a top differential form in a local uniformization
\begin{align*}
\varphi\colon \CC^n &\to \CC^n \\
(x_1,\ldots,x_n) &\mapsto (x_1^{m_i},\ldots,x_k^{m_k},x_{k+1},\ldots,x_n).
\end{align*}
We clearly have
$$
\varphi^* (dz_1 \wedge \ldots \wedge dz_n) = \prod_{i=1}^k m_i x_i^{m_i-1} dx_1 \wedge \ldots \wedge dx_n. 
$$
Thus we have to multiply with functions that along a ramification divisor $x_i=0$ are allowed to have poles of order at most $m_i-1$. On $X$, this means we have to multiply with functions that on the branch divisors $z_i=0$ have poles of order at most $\frac{m_i-1}{m_i}$. So $K_\orbi{X}$ locally is the pullback of the \emph{$\QQ$-divisor} $K_X + \Delta$~\cite[Prop.~4.4.15]{sasakian}.
\end{example}

\section{Orbisheaves}
\label{sec:orbisheaves}

We first introduce the notion of an orbisheaf following~\cite{MoerPronk} and~\cite[Def.~4.2.1]{sasakian}.

\begin{definition}
Let $\orbi{X}=(X,\orbi{U})$ be a complex orbifold. An \emph{orbisheaf} $\mathcal{F}$ on $\orbi{X}$ consists of a sheaf $\orbi{F}_i'$ over $U_i'$ for each orbifold chart $(U_i',G_i,\varphi_i,U_i)$ of $\orbi{X}$, such that for each injection $\lambda_{ji}:U_i' \to U_j'$ there is an isomorphism of sheaves $\orbi{F}(\lambda_{ij}:\orbi{F}_i' \to \lambda_{ji}^* \orbi{F}_j'$, which is functorial.
\end{definition}

We are mainly interested in sheaves of modules over a reasonable \emph{structure sheaf}, so first, we have to define such structure sheaf, see~\cite[Def.~4.2.2]{sasakian}.

\begin{definition}
The \emph{structure orbisheaf} $\orbi{O}_{\orbi{X}}$ is the orbisheaf consisting of structure sheaves $\orbi{O}_{U_i'}$ on each orbifold chart $U_i'$. On a complex orbifold $\orbi{X}$, by $\orbi{O}_{\orbi{X}}$ we will always denote the structure sheaf of holomorphic functions.
\end{definition}

It is clear that this definition neither will give us a sheaf on the underlying space $X$ nor it coincides with the holomorphic sections in the sense of Definition~\ref{def:orbsecdiff} of the trivial orbibundle, see Remark~\ref{rem:loctriv}. We have to use local $G_i$-invariant sections of such sheaves and glue them together over $X$~\cite[Lemma~4.2.4]{sasakian}. We will often work with these invariant sections of orbisheaves (or \emph{invariant} local sections of orbibundles, which do \emph{not in general} coincide with the \emph{equivariant} sections from Definition~\ref{def:orbsecdiff}). We will always denote sheaves on $X$ coming from invariant local sections of orbisheaves $\mathcal{F}$ by $\orbi{F}_X$. In particular $\left(\orbi{O}_\orbi{X}\right)_X \cong \orbi{O}_X$ holds for the structure sheaves.

Now recall that the functor $V \to V^G$ taking a vector space with an action of a \emph{finite} group $G$ to its $G$-invariant subspace is exact. This means in particular that for a \emph{coherent} orbisheaf $\orbi{F}$ of $\orbi{O}_\orbi{X}$-modules, the sheaf $\orbi{F}_X$ made up of (locally) $G_i$-invariant sections is a \emph{coherent} sheaf of $\orbi{O}_X$ modules. As exact sequences are preserved, it also makes sense to formulate orbisheaf cohomology, orbifold Dolbeault cohomology et cetera, see Section~\ref{sec:L2}.

\section{Orbimetrics}
\label{sec:orbimetrics}

In this section, we consider metrics on orbifolds, or \emph{orbimetrics}. By the preceding considerations, it is clear that these should be (invariant) metrics on the local uniformizations $U_i'$ of an orbifold $\orbi{X}=(X,\orbi{U})$.

\begin{definition}
Let $\orbi{X}=(X,\orbi{U})$ be a complex orbifold and $\orbi{E} \to \orbi{X}$ an orbibundle. A \emph{Hermitian orbimetric} on $\orbi{E}$ is a collection of Hermitian metrics $h_i'$ on the local uniformizations $E_i' \to U_i'$, such that all $h_i'$ are $G_i$-invariant and  all injections are Hermitian  isometries.
\end{definition}

We similarly can define Riemannian orbimetrics~\cite[Def.~4.2.11]{sasakian}, K\"ahler orbiforms, K\"ahler orbifolds~\cite[Def.~5.4.7]{MaMarinescu}, positive line orbibundles~\cite[Prop.~5.4.8]{MaMarinescu}, Hodge orbifolds et cetera - all as $G_i$-invariant objects on the local uniformizations $U_i'$ by the usual definitions.

On the other hand, if the underlying space $X$ of a complex orbifold $\orbi{X}=(X,\orbi{U})$ is smooth, then (usual) divisors or line bundles on $X$ can be pulled back to the local uniformizations and thus define orbibundles as we have seen in Example~\ref{ex:kandiv} in the case of the canonical divisor.

Now when the underlying space $X$ is even a K\"ahler manifold $(X,\omega)$ with K\"ahler form $\omega$ (in the usual sense), then Claudon~\cite[Prop.~2.1]{claudon} has constructed a K\"ahler \emph{orbiform} $\omega'$ out of $\omega$ in the following way.

\begin{example}
\label{ex:orbmet}
Let $\orbi{X}=(X,\Delta=\sum_{j=1}^m (1-1/m_j)\Delta_j)$ be a complex orbifold, such that $(X,\omega)$ is a K\"ahler manifold for some $(1,1)$-form $\omega$ on $X$. In a local uniformization $\varphi \colon \CC^n \cong U_i' \to U_i \cong \CC^n$, we can assume that on $U_i$ with coordinates $z_1,\ldots,z_n$, the K\"ahler form $\omega$ is given by
$$
\sum_{j=1}^n i \del \delbar \left|z_j \right|^{2} = i \sum_{j=1}^n   \dif z_j \wedge \dif \zbar_j.
$$
Analogous to Example~\ref{ex:kandiv}, the pullback under $\varphi$ is
$$
\varphi^*(\omega) = i \sum_{j=1}^n m_j^2\left|x_j\right|^{2(m_j-1)}  \dif x_j \wedge \dif \xbar_j,
$$
where $m_j=1$ if $x_j=0$ is not the restriction of a divisor $\Delta_j$.  
This form is clearly degenerate at the origin if  $\Delta \cap U_i \neq \emptyset$. In particular, it is no K\"ahler orbiform.

Now consider the global $(1,1)$-form $\omega_\Delta$ with values in $\orbi{O}_X(2\Delta)$ given by
$$
\omega_\Delta = \sum_{j=1}^k i \del \delbar \left|s_j \right|^{2/m_j}
$$
where $s_j \in \orbi{O}_X(\Delta_j)$ is a section defining $\Delta_j$. Locally we can assume that $s_j$ is just given by $z_j$, so the pullback by $\varphi$ is
$$
\varphi^*(\omega_\Delta) = \sum_{j=1}^k \dif x_j \wedge \dif \xbar_j.
$$
In general, $k < n$, so this form is denenerate as well.
Now combine these two to a form $\omega'=\omega+\omega_\Delta$. Then on the one hand, $\omega'$ is smooth on $X \setminus \supp(\Delta)$,  and for $c \in \RR_{>0}$ small enough, $\omega' \geq c\omega$ as currents. On the other hand, the pullback
$$
\varphi^*(\omega') = i \sum_{j=1}^k (1+m_j^2\left|x_j\right|^{2(m_j-1)}) \dif x_j \wedge \dif \xbar_j + \sum_{j=k+1}^m  \dif x_j \wedge \dif \xbar_j
$$
is a true K\"ahler form in the local uniformization $U_i'\cong \CC^n$.
\end{example}

What we need here is a stronger result. Consider the following situation: $\orbi{X}=(X,\Delta)$ is a complex orbifold with $X$ a manifold. Let $L$ be an ample line bundle on the manifold $X$. Then according to~\cite[Thm.~4.3.14]{sasakian} and the preceding paragraph therein, the \emph{first orbifold Chern class} of $L$ is just the usual first Chern class with respect to $X$. Thus $L$ (or the pullback to local uniformizations) defines an ample (or positive) line orbibundle. 

Now given a Hermitian positive line bundle $(L,h)$ on $X$ with curvature form $\Theta(L,h)$, such that $\omega = i\Theta(L,h)$ is a K\"ahler form, we want to \emph{explicitly construct} an orbimetric $H$ on $L$ as an orbibundle, such that $i\Theta(L,H)$ is a K\"ahler \emph{orbiform}. 

This directly leads to the notion of \emph{singular Hermitian metrics}, introduced in~\cite[Def.~2.1]{DemaillySingPos}.

\begin{definition}
Let $X$ be a complex manifold and $(L,h)$ a hermitian line bundle on $X$. A \emph{singular Hermitian metric} $H$ is a metric on $L$, given in a local trivialization $L \supseteq V \cong U \times \CC$  by $H=e^{-\phi}h$, where $\phi \in L^1_{\loc}(U,\RR)$ is a locally integrable function on $U$.
We call $(L,H)$ a \emph{singular Hermitian line bundle}. Due to~\cite[Def.~2.3.2]{MaMarinescu}, the \emph{curvature current} of $(L,H)$ is given by
$$
\Theta(L,H)=\Theta(L,h)+ \del \delbar \phi.
$$
\end{definition}

Thus we have the following.

\begin{proposition}
\label{prop:kaehlerorbi}
Let $\orbi{X}=(X,\Delta=\sum_{j=1}^m (1-1/m_j)\Delta_j)$ be a complex orbifold, where the underlying space $X$ is a manifold. For any $j=1,\ldots,m$, let $s_j \in \orbi{O}_X(\Delta_j)$ be a section defining $\Delta_j$. Let $(L,h)$ be a positive Hermitian line bundle on $X$. Then $(L,H)$ is a positive line orbibundle, where $H=e^{-\phi}h$ is given by
$$
\phi=\sum_{j=1}^{m} \left|s_j\right|^{2/m_j}.
$$ 

In particular, the form $\omega'$ given by
$$
\omega':= i \Theta(L,h)= i \Theta(L,h) + i\del \delbar \phi
$$
is a K\"ahler orbiform.
\end{proposition}

\begin{proof}
Since $(L,h)$ is positive, the form $\omega=i \Theta(L,h)$ is a K\"ahler form on the complex manifold $X$. On the other hand, $\varphi$ is chosen in such way, that  $i\del \delbar \phi$ coincides with $\omega_\Delta$ from Example~\ref{ex:orbmet}. Thus the computations from Example~\ref{ex:orbmet} verify the claim.
\end{proof}

Finally note that we can integrate $n$-forms by a partition of unity and by setting
$$
\int_{U_i} \sigma := \frac{1}{G_i} \int_{U_i'} \varphi^*(\sigma)
$$
in a local uniformization $(U_i',G_i,\varphi_i,U_i)$, see e.g.~\cite[Eq.~(4.2.2)]{sasakian}.
Thus if $(\orbi{L},h)$ is a Hermitian orbibundle on a complete K\"ahler orbifold $(\orbi{X},\omega')$, we have a scalar product 
$$
\left\langle s_1, s_2 \right\rangle := \int_X \left\langle s_1, s_2 \right\rangle_{h} dV_\omega
$$
for sections of $\orbi{L}$ and an associated $L^2$-norm $\left|\cdot\right|_h$, see~\cite[Sec.~5.4.2]{MaMarinescu}.

\section{The orbifold universal cover and the $\Gamma$-reduction}
\label{sec:orbicover}

\begin{definition}
\label{def:orbifund}
The \emph{orbifold fundamental group} of a geometric orbifold $\orbi{X}=(X,\Delta)$ is the quotient
$$
\pi_1(X,\Delta):=\pi_1(X \setminus \supp(\Delta)) / \langle \gamma_i^{m_i}, i \in I \rangle,
$$
where for each $i \in I$, $\gamma_i$ is a small loop around a general point of the divisor $\Delta_i$.
Associated to the orbifold fundamental group, there is the notion of \emph{orbifold universal cover} $\pi  \colon \widetilde{\orbi{X}} \to \orbi{X}$. It is a ramified Galois cover between complex analytic spaces, \'etale over $X \setminus \supp(\Delta)$.
\end{definition}

\begin{remark}
Let $\orbi{X}=(X,\Delta)$ be a smooth geometric orbifold. Then over a (sufficiently small) orbifold chart $(U_i',G_i',\varphi_i,U_i)$ of $X$ as in Remark~\ref{rem:geomorb}, the preimage under the orbifold universal cover $\pi \colon \widetilde{\orbi{X}} \to \orbi{X}$ has connected components $V_i$, such that $V_i$ has a local uniformization $(V_i',H_i',\psi_i,V_i)$ with $H_i$ a subgroup of $G_i$. In particular, since $G_i$ is abelian, $H_i$ is so as well and $V_i$ only has toric singularities. Locally, $\left. \pi \right|_{V_i} \colon V_i \to U_i$ is a quotient by $G_i / H_i$ and the lift $V_i' \to U_i'$ is just the identity~\cite[Rem.~1.2]{claudon}. So in a sense, the universal cover is locally trivial as we expect from a cover.
\end{remark}

As we mentioned before, the analogy between geometric and classical orbifolds not only holds if the underlying space is smooth. In particular, on the underlying space $X$ of a classical complex orbifold $\orbi{X}=(X,\orbi{U})$ one always can define a divisor $\Delta$, such that the geometric orbifold $(X,\Delta)$ has the canonical orbifold structure $\orbi{X}=(X,\orbi{U})$, see~\cite[p.~561]{BoyGalKoll}. In particular, this holds for the orbifold universal cover $\widetilde{\orbi{X}}$. But we do not need the structure of a geometric orbifold on $\widetilde{\orbi{X}}$ here.

An important observation for us will be that if $X$ is a complex analytic space,  $\Delta_1,\ldots,\Delta_m$ are smooth prime divisors on $X$ with normal crossings, and small loops $\gamma_i$ around general points of $\Delta_i$ are of finite order $m_i$ in $\pi_1(X \setminus(\Delta_1 \cup \ldots \cup \Delta_m))$, then 
$$
\pi_1(X \setminus(\Delta_1 \cup \ldots \cup \Delta_m)) = \pi_1\left(X,\sum_{i=1}^m \left(1-\frac{1}{m_i}\right)\Delta_i\right).
$$

Note that by the Hopf-Rinow-Theorem for orbifolds~\cite[Thm.~4.2.2]{Caramello}, the orbifold covers of a complete orbifold (with respect to an orbimetric $\omega'$, cf. Section~\ref{sec:orbimetrics}), are complete with respect to the pullback metric (since orbifold geodesics can be lifted).  In particular, the orbifold universal cover of a compact orbifold with a Hermitian orbimetric is complete with respect to the pullback orbimetric.

An important ingredient for us is the \emph{$\Gamma$-reduction} or \emph{Shafarevich map}. This construction has been introduced by Koll\'ar for proper normal projective varieties~\cite[Def.~1.4]{KollarShaf2} and independently by Campana for compact K\"ahler manifolds~\cite[Thm.~3.5,Def.~3.8]{CampGamma1}. Formulated on the universal cover $\widetilde{X}$ of a compact K\"ahler manifold $X$, it says that there is a unique almost holomorphic fibration $\widetilde{\gamma} \colon \widetilde{X} \dasharrow \Gamma(\widetilde{X})$, such that any compact irreducible subvariety of $\widetilde{X}$ through a \emph{very general point} $x \in \widetilde{X}$ is contained in the fiber $\widetilde{\gamma}^{-1}(\widetilde{\gamma}(x))$. The general fibers of $\widetilde{\gamma}$ are exactly the maximal compact subvarieties of $\widetilde{X}$. The action of $\pi_1(X)$ on $\widetilde{X}$ descends to $\Gamma(\widetilde{X})$ and thus by quotienting induces an almost holomorphic fibration $\gamma \colon X \dasharrow \Gamma(X)$, of which the fibers are the maximal subvarieties with \emph{finite} fundamental group. In turn, the connected components of the preimages of such fibers are exactly the fibers of $\widetilde{\gamma}$.

This concept has been generalized by Claudon in~\cite{claudon} to smooth geometric orbifolds - using K\"ahler orbiforms as in Example~\ref{ex:orbmet} -, see also~\cite[Sec.~12.5]{CampSecOrbi}. We have the following~\cite[Thm.~0.2]{claudon}.

\begin{theorem}
\label{thm:G-red}
Let $\orbi{X}=(X,\Delta)$ be a compact smooth geometric K\"ahler orbifold and $\pi \colon \widetilde{\orbi{X}} \to \orbi{X}$ its orbifold universal cover. There  are almost holomorphic fibrations
$
\widetilde{\gamma} \colon \widetilde{\orbi{X}} \dasharrow \Gamma(\widetilde{\orbi{X}})
$
and
$
\gamma \colon \orbi{X} \dasharrow \Gamma(\orbi{X})
$
, such that the diagram
$$
 \xymatrix@C=3em{ 
\widetilde{\orbi{X}} \ar@{-->}[r]^{\widetilde{\gamma}}  \ar[d]^{/ \pi_1(\orbi{X})}
&
\Gamma(\widetilde{\orbi{X}}) \ar[d]^{/ \pi_1(\orbi{X})}  \\ 
\orbi{X} \ar@{-->}[r]_{\gamma} 
&
\Gamma(\orbi{X})
}
$$
commutes and the following hold:
\begin{enumerate}
\item If $V \subseteq X$ is a smooth subvariety meeting $\Delta$ transversally, such that the image of $\pi_1(V,\left.\Delta\right|_V)$ in $\pi_1(X,\Delta)$ is finite, and $V$ meets the fiber of $\gamma$ through a very general point, then $V$ is contained in this fiber. 
\item Every compact irreducible subvariety of $\widetilde{\orbi{X}}$ through a very general point $x \in \widetilde{\orbi{X}}$ is contained in the fiber $\widetilde{\gamma}^{-1}(\widetilde{\gamma}(x))$.
\item There exist open subsets $X^0 \subset X$ and $\Gamma(\orbi{X})^0 \subset \Gamma(\orbi{X})$, such that $\left. \gamma\right|_{X^0} \colon X^0 \to \Gamma(\orbi{X})^0$ is a proper holomorphic, topologically locally trivial fibration.
\end{enumerate}
\end{theorem}

\begin{remark}
Theorem 0.2 of~\cite{claudon} is formulated only on the universal cover, while~\cite[Thm.~12.23]{CampSecOrbi} is formulated on the orbifold $\orbi{X}$ itself. The connection between the both is~\cite[Le.~2.2]{claudon}. The third item has not been formulated in the orbifold case, but the argument at the end of the proof of Proposition 2.4 in~\cite{KollarShaf2} works here as well.
\end{remark}

\section{Dolbeault and $L^2$-cohomology for K\"ahler orbifolds}
\label{sec:L2}

Following~\cite[Sec.~5]{OrbifoldSurvey}, we can define orbifold Dolbeault cohomology for  complete K\"ahler orbifolds $(\orbi{X}=(X,\orbi{U}),\omega)$ in the following way.
Denote by $\Omega_X^{p,q}$ the sheaf of  $(p,q)$-orbiforms, defined by the usual $(p,q)$-forms on the local uniformizations. The locally invariant sections define the sheaf $\Omega^{p,q}_X$ on the udnerlying space $X$ and we denote the space of global sections by $\Omega^{p,q}_X(X)$. The exterior derivative and the Dolbeault operators $\dif= \del + \delbar$ are well defined, with
$$
\del\colon \Omega^{p,q}_X \to \Omega^{p+1,q}_X, \qquad
\delbar\colon \Omega^{p,q}_X \to \Omega^{p,q+1}_X.
$$

\begin{definition}
The \emph{$(p,q)$-th orbifold Dolbeault cohomology group} is defined by
$$
H^{p,q}(X) := \frac{\ker(\delbar\colon \Omega_X^{p,q}(X) \to \Omega_X^{p,q+1}(X))}{\im(\delbar\colon \Omega_X^{p,q-1}(X) \to \Omega_X^{p,q}(X))}.
$$ 

\end{definition}

If $\orbi{E} \to \orbi{X}$ is a holomorphic orbibundle, then one can similarly define the Dolbeault complex $(\Omega_X^{p,q}(X,\orbi{E}),\delbar^\orbi{E})$ of $(p,q)$-orbiforms with values in $\orbi{E}$ as well as Dolbeault cohomology groups $H^{p,q}(X,\orbi{E})$. Then the \emph{Dolbeault isomorphism for orbifolds} holds, see~\cite[Sec.~5.4.2]{MaMarinescu}.
Now let $\orbi{E}$ be endowed with a (smooth or singular) Hermitian orbimetric $h$.
Following~\cite[Eq.~(B.4.12)]{MaMarinescu},  we define the $L^2$-spaces
$$
L^2_{p,q}(X,\orbi{E}):=\{s \in \Omega_X^{p,q}(X,\orbi{E});~ \int_X \left| s\right|^2_h dV_\omega < \infty \},
$$
and the $L^2$-Dolbeault cohomology groups by 
$$
H_{(2)}^{p,q}(X,\orbi{E}) := \frac{\ker(\delbar^\orbi{E}) \cap L^2_{p,q}(X,\orbi{E})}{\im(\delbar^\orbi{E}) \cap L^2_{p,q}(X,\orbi{E}) }.
$$ 
Well-definedness follows from~\cite[Sec.~C.3]{Ballmann}, which can be directly transferred to complete K\"ahler orbifolds.

\section{$L^2$-vanishing for orbifolds}
\label{sec:orbivan}

The singular Hermitian metrics from Section~\ref{sec:orbimetrics} will be more useful to us than just for constructing K\"ahler orbiforms from positive line bundles on the underlying space. For a singular Hermitian line bundle $(L,H)$ with $H=e^{-\phi}h$ on a complex manifold $X$, there is the notion of the \emph{$L^2$-sheaf} $\orbi{L}^2(L,H)$ of locally square-integrable functions with respect to $H$, given by
$$
\orbi{L}^2(L,H)(U) = \{ \sigma \in  \Gamma(U,L);~ \left|\sigma\right|_h^2 e^{-\phi} \in L_{\loc}^1(U)\},
$$
see~\cite[Eq.~(3.1)]{TakayamaNonvanishing}. In particular, the function $\phi$ defines a singular Hermitian metric $e^{-\phi} z \zbar$ on the trivial line bundle $X \times \CC$. This leads us to the definition of the \emph{multiplier ideal sheaf} $\orbi{I}(\phi):=\orbi{L}^2(X \times \CC,e^{-\phi})$. In particular, $\orbi{L}^2(L,H) = L \otimes \orbi{I}(\phi)$. Note that the functions $\phi$ may only be given locally, so in this notation $\phi$ can rather be seen as a collection of locally defined functions. On the other hand, it may still be possible to express $\phi$ globally by certain sections as e.g. in Proposition~\ref{prop:kaehlerorbi}.

A \emph{plurisubharmonic} or shortly \emph{psh} function is defined by certain semicontinuity properties, see e.g.~\cite[Def.~(1.4)]{DemL2Van}. We will use the following characterization from~\cite[Prop.~B.2.10,~B.2.16]{MaMarinescu}, which is much more immediate in our setting.

\begin{definition}
Let $X$ be a complex analytic manifold. A function $\phi: X \to \RR$ is called \emph{plurisubharmonic} or \emph{psh}, if $i\del \delbar \phi$ is a semipositive form. It is called \emph{strictly psh}, if $\phi \in L^1_\loc(X)$ and $i\del \delbar \phi$ is (strictly) positive.
\end{definition}

The point is that obviously on the one hand, a singular Hermitian metric $H$ on a positive Hermitian line bundle $(L,h)$ defined by a psh function $\phi$  gives a \emph{positive} $(1,1)$-form $\omega=i\Theta(L,H)$. 

On the other hand, we have the \emph{Nadel coherence theorem}~\cite[Prop.~(5.7)]{DemL2Van}, stating that $\orbi{I}(\phi)$ is a \emph{coherent sheaf} if $\phi$ is psh. This can be easily transformed to the orbifold setting. First let us define the analogue of the multiplier ideal sheaf following~\cite[Def.~5.2.9]{sasakian}.

\begin{definition}
Let $\orbi{X}=(X,\Delta)$ be a complex orbifold and let $(L,H=he^{-\phi})$ be a singular Hermitian orbibundle on $\orbi{X}$. The \emph{multiplier ideal orbisheaf} $\orbi{I}_X(\phi)$ is the orbisheaf defined on local uniformizations  $U_i' \to U_i$ by
$$
\orbi{I}_X(\phi)(U_i')=\left\lbrace f \in \orbi{O}_{\orbi{X}}^{G_i}(U_i');~ |f|^2e^{-\phi} \in L^1_{\loc}(U_i')\right\rbrace.
$$
\end{definition}

The orbifold version of Nadel's coherence theorem follows from the standard version since the functor taking $G_i$-invariant sections is exact by finiteness of $G_i$. Thus we have:

\begin{theorem}[Nadel's coherence theorem for orbifolds]
Let $\orbi{X}=(X,\Delta)$ be a complex orbifold and $(L,H=he^{-\phi})$ be a singular Hermitian orbibundle on $\orbi{X}$. Then the (pushforward of the) multiplier ideal orbisheaf $\orbi{I}_X(\phi)$ is a coherent sheaf of $\orbi{O}_X$-modules on $X$.
\end{theorem}

The next step to go now is the \emph{Nadel vanishing theorem}. The orbifold version is the following~\cite[Thm.~6.5]{DemKoll}.

\begin{theorem}[Nadel's vanishing theorem for orbifolds]
\label{thm:nadelvan}
Let $(\orbi{X},\omega')$ be a K\"ahler orbifold, that is a complex orbifold $\orbi{X}=(X,\Delta)$ 
with a K\"ahler orbiform $\omega'$. Let $(L,H=he^{-\phi})$ be a singular Hermitian orbibundle on $\orbi{X}$, where $h$ is a smooth Hermitian orbimetric on $L$. Assume that there exists a constant $c \in \RR_{>0}$, such that $i\Theta(L,H)\geq c\omega'$. If $K_\orbi{X} \otimes L$ is an \emph{invertible} sheaf on $X$, then 
$$
H^q(X,K_\orbi{X} \otimes L \otimes \orbi{I}_X(\phi))=0 ~\mathrm{for}~ q \geq 1.
$$
\end{theorem}

Several things have to be noted. First, as Koll\'ar and Demailly stress in the paragraph after~\cite[Thm.~6.5]{DemKoll}, for the orbifold version, it is really necessary that $K_\orbi{X} \otimes L$ is an invertible sheaf on $X$. This is because on the one hand, the above tensor product $K_\orbi{X} \otimes L \otimes \orbi{I}_X(\phi)$ means first taking the usual tensor product on local uniformizations $U_i'$, then taking $G_i$-invariant sections, and finally the direct image sheaves on $U_i$. On the other hand, the statement is obtained by $L^2$-estimates - with respect to the weight $e^{-\phi}$ - of sections of $K_\orbi{X} \otimes L$ on $X \setminus \supp(\Delta)$.

It turns out that for us these $L^2$-estimates are even more important than the statement of Theorem~\ref{thm:nadelvan}. As they are not explicitly stated in~\cite{DemKoll}, we refer to~\cite[Cor.~(5.3)]{DemL2Van}. See also~\cite[Thm.~4.4]{TianXu} and the subsequent paragraph therein for the orbifold case.

\begin{proposition}
\label{prop:delbarest}
Let $(\orbi{X},\omega')$ be a complete K\"ahler orbifold. Let $(L,H=he^{-\phi})$ be a singular Hermitian orbibundle on $\orbi{X}$, where $h$ is a smooth Hermitian orbimetric on $L$. Assume that there exists a constant $c \in \RR_{>0}$, such that $i\Theta(L,H)\geq c\omega'$. Then for any $\delbar$-closed form $g \in L_{p,q}^2(\orbi{X}, L)$,  there is a form $f \in L_{p,q-1}^2(\orbi{X},L)$, with $\delbar f=g$ and
$$
\int_X \left|f\right|_H^2 \dif V_{\omega'} \leq \frac{1}{qc} \int_X \left|g\right|_H^2 \dif V_{\omega'}.
$$
\end{proposition}

\section{Maximal compact subspaces of orbifold universal covers}
\label{sec:maxsubsp}

This section is merely a translation of~\cite[Sec.~4]{TakayamaNonvanishing} to the orbifold case. Following~\cite[Sec.~3B]{TakayamaNonvanishing}, for any complex analytic space $X$, by a \emph{subvariety} $W \subseteq X$, we mean an irreducible reduced complex subspace. By a \emph{maximal compact subspace} $Z \subseteq X$, we mean a not necessarily reduced nor irreducible compact subspace, such that every subvariety $W \subseteq X$ with $Z \cap W \neq \emptyset$ is contained in $Z$.

We have the following (compare~\cite[Prop.~4.1]{TakayamaNonvanishing}).

\begin{proposition}
\label{prop:compsubs}
Let $\orbi{X}=(X,\Delta)$ be a smooth complex compact orbifold and $\widetilde{\orbi{X}}$ its orbifold universal cover with underlying space denoted by $\widetilde{X}$. Let $(L,h)$ be a positive Hermitian line orbibundle on $\orbi{X}$, such that $(\orbi{X},\omega=i \Theta(L,h))$ becomes a complete K\"ahler orbifold. Denote the pullbacks of $L$, $H$,  and $\omega$ by $\widetilde{L}$, $\widetilde{h}$ , and $\widetilde{\omega}$ respectively. Let moreover $Z\subseteq \widetilde{X}$ be a connected maximal compact subspace and $N \in \ZZ_{\geq 1}$.  Then there exists  a singular Hermitian metric $H$ on $\widetilde{L}$ with the following properties:
\begin{enumerate}
\item $i\Theta(\widetilde{L},H) \geq (1-1/N)\widetilde{\omega}$ as currents.
\item There exists an open neighbourhood $U$ of $Z$, such that  
$$
U \cap \supp (\mathcal{O}_{\widetilde{X}}/\mathcal{I}_{\widetilde{X}}(H))= Z
$$
 and $\left.\mathcal{I}_{\widetilde{X}}(H)\right|_{U} \subset \left.\mathcal{I}_{Z}\right|_{U}$.
\item  There exists a positive constant $c_0$, such that $h \leq c_0 H$.
\end{enumerate}
\end{proposition}

To prove the proposition, we need the following three lemmata.

\begin{lemma}
\label{le:firstle}
Let $\orbi{X}$, $\widetilde{\orbi{X}}$, and $(L,h)$ be as in Proposition~\ref{prop:compsubs}. Let $\{x_i\}_{i \in \NN}$ be a discrete sequence of points in $\widetilde{X}$ with no accumulation point. Then there exists a subsequence $\{x_{i_k}\}_{k \in \NN}$ and a positive integer $m_0$, such that for any $m \in \ZZ_{>m_0}$ and  $\ell\in \ZZ_{>0}$, the evaluation map
$$
H^0_{(2)}\left(\widetilde{X},\widetilde{L}^{\otimes m}\right) \to \bigoplus_{k=1}^\ell \mathcal{O}_{\widetilde{X}} / \mathcal{M}_{\widetilde{X},x_{i_k}}
$$
is surjective.
\end{lemma}

\begin{proof}
This is basically the proof of~\cite[Le.~4.2]{TakayamaNonvanishing} translated to the orbifold setting.
Since $\{x_i\}_{i \in \NN}$ has no accumulation point, we can take a subsequence, which by abuse of notation we again denote by $\{x_i\}_{i \in \NN}$, such that there exists $\epsilon >0$ with $\dist_{\omega}(x_i,x_j)> \epsilon \diam(X,\omega)$ for $i \neq j$. 

Now take $\epsilon \in \RR_{>0}$ and consider local uniformizations $(U_i'= B_\epsilon(0) \subseteq \CC^n,G_i,\varphi_i,U_i)$ around $x_i$, such that $\varphi_i(0)=x_i$ for any $i \in \NN$. By the \emph{bounded geometry} of $\widetilde{\orbi{X}}$ as orbifold cover of the compact orbifold $\orbi{X}$, compare~\cite[Le.~2.1]{claudon}, there is a constant $c \in \RR_{>0}$ such that for any $i$ and the standard metric $g_i$ on $U_i'$, we have
$$
\frac{1}{c}g_i < \omega < c g_i.
$$
Now take a smooth $U(n)$-invariant cutoff function $\rho:B_\epsilon(0) \to [0,1]$ with compact support satisfying $\rho \equiv 1$ on $B_{\epsilon/3}(0)$ and $B_{\epsilon}(0) \setminus B_{2\epsilon/3}(0)$. Define
$$
\phi:= \sum_{i \in \NN} n \rho(z) \log \sum_{j=1}^n \|z_j\|^2 \in L_\loc^1(\widetilde{\orbi{X}},\RR).
$$
It is obvious that $\phi$ is $U(n)$- and thus $G_i$-invariant for all $i \in \NN$. The multiplier ideal orbisheaf $\orbi{I}_X(\phi)$ defines a complex subspace of $X$, which is exactly $\{x_i\}_{i \in \NN}$. 

Now since $(\widetilde{L},\widetilde{h})$ is positive, there is $a_0 \in \ZZ_{>1}$, such that $i \del \delbar \log \widetilde{\omega}^n + a_0 \widetilde{\omega}$ is positive, that is $K_{\widetilde{\orbi{X}}}^{\otimes(-1)} \otimes \widetilde{L}^{\otimes a_0}$ is positive. Moreover, due to the definition of $\phi$ and the bounded geometry property from above, there is $b_0 \in \ZZ_{>1}$, such that $-b_0 \omega < i \del \delbar \phi < b_0 \omega$, see~\cite[Le.~2.3]{TakRems}. 
The space 
$$
H_{(2)}^1(X, \orbi{L}^2(L^{\otimes m},e^{-\phi}h^{\otimes m}))=H_{(2)}^{1}(X, K_{\widetilde{\orbi{X}}} \otimes K_{\widetilde{\orbi{X}}}^{\otimes(-1)} \otimes \widetilde{L}^{\otimes m} \otimes \orbi{I}_X(\phi))
$$ 
vanishes for every $m>m_0:=a_0+b_0$ due to Proposition~\ref{prop:delbarest}. This means that the map  
$$
H^0_{(2)}\left(\widetilde{X},\widetilde{L}^{\otimes m}\right) \to \bigoplus_{k=1}^\ell \mathcal{O}_{\widetilde{X}} / \mathcal{M}_{\widetilde{X},x_{i_k}}
$$
indeed is surjective for all $\ell\in \ZZ_{>0}$.
\end{proof}

\begin{lemma}
\label{le:secondle}
Let $\orbi{X}$, $\widetilde{\orbi{X}}$, and $(L,h)$ be as in Proposition~\ref{prop:compsubs}.
Let $Z \subseteq \widetilde{X}$ be a compact complex subspace, $Y \subseteq \widetilde{X}$ be a positive-dimensional non-compact subvariety, and $N$ a positive integer. Then there exist a positive integer $m$ and an $L^2$-section $s \in H^0_{(2)}\left(X,\widetilde{L}^{\otimes m} \otimes \mathcal{I}_Z^{mN}\right)$, such that $\left.s\right|_Y \neq 0$.
\end{lemma}

\begin{proof}
Since $Y$ is non-compact, we can take a sequence of points $\{x_i\}_{i \in \NN}$ in $Y$ with no accumulation points in $\widetilde{X}$ (since $Y$ is closed). By Lemma~\ref{le:firstle}, we can take a subsequence, which again we denote by $\{x_i\}_{i \in \NN}$, such that there exists $m \in \NN$ with the  map $
H^0_{(2)}\left(X,\widetilde{L}^{\otimes m}\right) \to \bigoplus_{i=1}^\ell \mathcal{O}_X / \mathcal{M}_{X,x_{i}}
$ being surjective for all $\ell \in \NN$. We consider now the exact sequence
$$
 \xymatrix@C=1.5em{ 
0 \ar[r]
&
H^0_{(2)}\left(\widetilde{X}, \widetilde{L}^{\otimes m} \otimes \mathcal{I}_Z^{mN}\right) \ar[r]
&
H^0_{(2)}\left(\widetilde{X}, \widetilde{L}^{\otimes m}\right) \ar[r]
&
H^0\left(\widetilde{X}, \widetilde{L}^{\otimes m} \otimes \mathcal{O}_X / \mathcal{I}_Z^{mN}\right).}
$$
The last term  has dimension $d \in \ZZ_{\geq 0}$, since $Z$ is compact. Using Lemma~\ref{le:firstle}, we choose $\ell \in \ZZ_{>d}$ and $L^2$-sections $\{s_i\}_{i=1}^{\ell} \subset H^0_{(2)}\left(X, \widetilde{L}^{\otimes m}\right)$ such that $s_i(x_i) \neq 0$ and $s_i(x_j)=0$ for $1\leq i \neq j \leq \ell$. Since $\ell > d$, there is a linear combination $s$ of the $s_i$'s which is a  nontrivial $L^2$-section in $H^0_{(2)}\left(X, \widetilde{L}^{\otimes m} \otimes \mathcal{I}_Z^{mN}\right)$. Also $\left. s\right|_Y$ is not the zero section over $Y$ due to the choice of the $s_i$'s.
\end{proof}

Let $\alpha \in \QQ_{>0}$. Following~\cite{TakayamaNonvanishing}, by a \emph{multivalued $L^2$-section} of  $\widetilde{L}^{\otimes \alpha}$, we denote a section $s$ of $\widetilde{L}^{\otimes \alpha}$, such that there is $p \in \ZZ_{>0}$ with $p\alpha \in \ZZ$ and $s^p \in H_{(2)}^0(\widetilde{X},\widetilde{L}^{\otimes p\alpha})$. We can then define the pointwise length 
$$
\left|s\right|_{\widetilde{h}}:=\left(\widetilde{h}^{\otimes p\alpha}(s_i^p,s_i^p)\right)^{1/(2p)}
$$ 
and the zero locus $(s)_0=(s^p)_0$ of such sections. If $k \in \ZZ_{>0}$ and $s=\{s_i\}_{i=1,\ldots,k}$ is a finite number of multivalued $L^2$-sections of $L^{\otimes \alpha}$, we denote $\left|s\right|:= \sum_{i=1}^k \left|s_i\right|_h^2$ and $(s)_0:= \bigcap_{i=1}^k (s_i)_0$. Moreover, we define a multiplier ideal sheaf for $s$ by 
$$
\orbi{I}(s):=\orbi{L}\left(\orbi{O}_{\widetilde{\orbi{X}}},(\left|s\right|^2)^{-1}\right).
$$

By~\cite[Le.~2.4]{TakRems}, for $k \in \ZZ_{>0}$, the pointwise length of an $L^2$-section $s \in H_{(2)}^0(\widetilde{X},\widetilde{L}^{\otimes k})$ tends to zero at infinity. In particular, in the above setting, if $s^p \in H_{(2)}^0(\widetilde{X},\widetilde{L}^{\otimes p\alpha})$, then $s^q \in H_{(2)}^0(\widetilde{X},\widetilde{L}^{\otimes q\alpha})$ as well for $q \in \ZZ_{\geq p}$.

\begin{lemma}
\label{le:thirdle}
Let $\orbi{X}$, $\widetilde{\orbi{X}}$, and $(L,h)$ be as in Proposition~\ref{prop:compsubs}.
Let $Z \subseteq \widetilde{X}$ be a compact complex subspace, $U \subseteq \widetilde{X}$ a relatively compact open subset, and $N$ a positive integer. Then there is some $k \in \ZZ_{>0}$ and multivalued $L^2$-sections $s=\{s_i\}_{i=1,\ldots,k}$ of $\widetilde{L}^{\otimes 1/N}$ such that the following hold:
\begin{enumerate}
\item The set of common zeros $(s)_0$ of the $s_i$ has no non-compact irreducible component that intersects $U$. 
\item The multiplier ideal sheaf $\mathcal{I}(s)$ is contained in the ideal sheaf $\mathcal{I}_{Z}$.
\end{enumerate}
\end{lemma}

\begin{proof}
First, note that there exists a positive integer $q$, such that, for every $m \in \ZZ_{>0}$, we have 
$$
\orbi{L}\left(\orbi{O}_{\widetilde{\orbi{X}}},(\left|s\right|^2)^{-1/(mN)}\right) \subset \mathcal{I}_Z
$$ 
for any set of sections $s=\{s_i\}_{i=1,\ldots,k} \subset H^0_{(2)}(X, L^{\otimes m} \otimes \mathcal{I}_{\red Z}^{mqN})$, where $\red Z$ is the reduced structure of $Z$. 

Fix such an integer $q$. By Lemma~\ref{le:secondle}, there is $m_1 \in \ZZ_{>0}$ and a nonzero $L^2$-section $s_1' \in H^0_{(2)}(X, \widetilde{L}^{\otimes m_1} \otimes \mathcal{I}_{\red Z}^{m_1qN})$. We set $s_1:={s'_1}^{1/(m_1N)}$, which is a multivalued section of $L^{\otimes 1/N} \otimes \mathcal{I}_{\red Z}^{q}$. Now if there is no non-compact irreducible component $Y$ of $(s_1)_0$ intersecting $U$, set $s:=\{s_1\}$. If there \emph{is} a non-compact irreducible component $Y$ of $(s_1)_0$ intersecting $U$, then apply Lemma~\ref{le:secondle} for $Z$ and this $Y$. It follows that there is an $L^2$-section $s_2$ of $L^{\otimes 1/N}$ such that $s_2^{m_2N} \in H^0_{(2)}(X, L^{\otimes m_2} \otimes \mathcal{I}_{\red Z}^{m_2qN})$ for a positive integer $m_2$, such that $\left.s_2\right|_Y$ is not the zero section. Now we pass to $(s_1)_0 \cap (s_2)_0$ and check if there is a non-compact irreducible component intersecting $U$. If yes, proceed again with Lemma~\ref{le:secondle} to obtain a section $s_3$ et cetera. Since $U$ is relatively compact, after a finite number of steps, we have $L^2$-sections $s_1,\ldots,s_k \in \widetilde{L}^{\otimes 1/N}$ satisfying the requirements of the lemma. 
\end{proof}

\begin{proof}[Proof of Proposition~\ref{prop:compsubs}]
First let $U$ be a relatively compact open neighbourhood of the connected maximal compact subspace $Z$. Apply Lemma~\ref{le:thirdle} on this $U$ and $Z$, $N$ from the proposition. We use the $L^2$-sections $s=\{s_i\}_{i=1}^{m}$ of $\widetilde{L}^{\otimes 1/N}$ from the lemma to construct a singular Hermitian metric 
$$
H:=h^{\otimes (1-1/N)} \frac{h^{\otimes 1/N}}{\left| s \right|^2}
$$
of $\widetilde{L}$, having the properties:
\begin{enumerate}
\item $i\Theta(\widetilde{L},H) = i\Theta(\widetilde{L}^{\otimes (1-1/N)}, h^{\otimes (1-1/N)}) + i\Theta(h^{\otimes 1/N}(\left| s \right|^2)^{-1}) \geq (1-1/N)\widetilde{\omega}$.
\item $\mathcal{I}_{\widetilde{X}}(H) =\mathcal{I}(s)$, so $\left.\mathcal{I}(H)\right|_U$ defines a compact complex subspace of $U$ containing $Z$.
\item There is an upper bound for $\left|s\right|^2$, since by~\cite[Le.~2.4]{TakRems}, $\left|s\right|$ tends to zero  at infinity. So there is a positive constant such that $h \leq c_0 H$.
\end{enumerate} 
Since $Z$ is a \emph{maximal} compact subspace, we have $
U \cap \supp (\mathcal{O}_{\widetilde{X}}/\mathcal{I}_{\widetilde{X}}(H))= Z
$, and the proposition is proven
\end{proof}

\section{Proof of Theorem~\ref{thm:loctoglob}}
\label{sec:loctoglob}

In this section we prove Theorem~\ref{thm:loctoglob}, which makes up the local-to-global part of the induction in the proof of our main theorems. First we recall the necessary definitions.

\subsection*{Definitions}
We call a pair $(Y,D)$ of a normal complex algebraic variety $Y$ and an effective $\QQ$-divisor $D=\sum d_j D_j$ on $Y$ with $K_Y+D$ being $\QQ$-Cartier a \emph{log pair}.
For a log pair $(Y,D)$, in the following we will often decompose $D=D'+D''$, where $D''\geq 0$, and $D'=\sum(1-1/m_i)D_i$  is a sum of prime divisors $D_i$, whose coefficients satisfy $m_i \in \ZZ_{>1}$.

We say that a birational divisorial contraction $f: X \to Y$ is a \emph{log resolution} of the pair $(Y,D)$, if $X$ is smooth and  $f_*^{-1}\supp(D) \cup E_1 \cup \ldots \cup E_k$ is a simple normal crossing divisor, where $E_i$, $i=1,\ldots,k$ are the $f$-exceptional prime divisors.

We call a log pair $(Y,D)$ \emph{Kawamata log terminal} or \emph{klt} shortly, if $0<d_i<1$ and there exists a  log resolution $f: X \to Y$, such that we can write
$$
K_X + f_*^{-1} (D)  + \sum E_i = f^*(K_Y + D) + \sum a_i E_i,
$$
where the $a_i$, which we call \emph{log-discrepancies}, are greater than zero. Note that $f_*^{-1} (D)=\sum d_i f_*^{-1} (D_i)$.

We call a projective variety $Y$ \emph{weakly Fano}, if there exists an effective $\QQ$-divisor $D=\sum d_j D_j$ on $Y$, such that $(Y,D)$ is klt and $-(K_Y+D)$ is big and nef. 

The statement of Theorem~\ref{thm:loctoglob} to prove now is the following: assume that $n$-dimensional klt-singularities have \emph{finite regional fundamental group}, then $n$-dimensional weakly Fano pairs $(Y,D'+D'')$ have \emph{finite orbifold fundamental group} $\pi_1(Y_\sm,D')$.

\subsection*{Compact orbifolds supported on a log resolution}
The proof of the above statement relies on the following two propositions, which essentially say that for a log resolution $f:X \to Y$ of a weakly Fano pair $(Y,D'+D'')$ with exceptional divisor $E$, for any admissible $\QQ$-divisor $\Delta$ supported on $\supp(E) \cup \supp(f_*^{-1} D')$, the smooth geometric orbifold $(X,\Delta)$ has finite fundamental group.  

\begin{proposition}
\label{prop:nontrivsec}
Let $(Y,D'+D'')$ be a weakly Fano pair, with $D'=\sum (1-1/e_j)D_j$ for some $e_j \in \ZZ_{>1}$ and $D'' \geq 0$. Let $f:X \to Y$ be a log resolution with exceptional prime divisors $E_1,\ldots,E_k$. For arbitrary $m_i \in \ZZ_{>0}$, consider the smooth geometric orbifold $\orbi{X}:=(X,\Delta:=f_*^{-1} D' + \sum (1-1/m_i)E_i)$. Define a divisor 
$$
L:= -f^*(K_Y + D) - \sum_{-1 <c_i<0} c_i E_i + \sum_{0 \leq c_i} \left(\left\lceil c_i \right\rceil - c_i \right) E_i + f_*^{-1}D'',
$$
where $c_i:=a_i-1/m_i$ and the $a_i > 0$ are the log-discrepancies. Consider $L$ as an orbibundle on $\orbi{X}$. Then the orbifold universal cover $\pi:\widetilde{\orbi{X}} \to \orbi{X}$ has a nontrivial $L^2$-section
$$
\nu \in H_{(2)}^0(\widetilde{\orbi{X}},K_{\widetilde{\orbi{X}}} \otimes \pi^*L).
$$
\end{proposition}

\begin{proof}
Consider a pair $(Y,D'+D'')$, a log resolution $f:X \to Y$ and arbitrary $m_i \in \ZZ_{>0}$  as in the proposition. Set $\Delta:=f_*^{-1} D' + \sum (1-1/m_i)E_i$. Consider the smooth geometric projective orbifold $\orbi{X}=(X,\Delta)$. Then the orbifold canonical divisor of $\orbi{X}$ is defined by $K_{\orbi{X}}:= K_X +\Delta$, see Section~\ref{sec:orbibundles}. By the above ramification formula, we can write
$$
K_{\orbi{X}} +  f_*^{-1}D''   = f^*(K_Y + D) + \sum c_i E_i ,
$$
where $c_i:=a_i-1/m_i>-1$ and $-f^*(K_Y + D)$ is big and nef. Now define
$$
\Delta':=\sum_{-1 < c_i <0} (-c_i) E_i + \sum_{0 \leq c_i } (\left\lceil c_i \right\rceil -c_i)E_i + f_*^{-1}D'' ,
\quad
E:=\sum_{0 \leq c_i } \left\lceil c_i \right\rceil E_i, 
$$
and $L:= -f^*(K_Y + D) + \Delta'$. With these definitions, the ramification formula becomes
$
K_{\orbi{X}} + L = E
$.
Now since $L$ is the sum of a big and nef and a simple normal crossing $\QQ$-divisor with coefficients strictly between $0$ and $1$, $L=A+\Delta''$, where $A$ is an ample $\QQ$-divisor and $\Delta''=\Delta'+N$, where $N$ is a very small effective $\QQ$-divisor. This means in particular that the pair $(X,\Delta'')$ is klt. 

Since $A$ is an ample $\QQ$-divisor, there is a positive integer $a$, such that, by Proposition~\ref{prop:kaehlerorbi}, $A^{\otimes a}$ is a positive line \emph{orbibundle} with orbimetric $h_A$.  Denote by $\omega:=i\Theta(A^{\otimes a},h_{A^{\otimes a}})$ the corresponding K\"ahler orbiform. Then $(\orbi{X}, \omega)$ is a compact K\"ahler orbifold.

Now following~\cite[Sec.~3B]{TakayamaNonvanishing}, take a multivalued canonical section $\sigma_{\Delta''}$ of $\orbi{O}_X(\Delta'')$, that is $m\Delta''$ is an integral effective divisor for some positive integer $m$, and $\divisor(\sigma_{\Delta''}^m)=m\Delta''$. Take in addition a Hermitian metric $h_{m\Delta''}$ of  $\orbi{O}_X(m\Delta'')$ and define a function $\left|\Delta''\right|:=\left|\sigma_{\Delta''}^m\right|_{h_{m\Delta''}}^{1/m}$. 
Since the pair $(X,\Delta'')$ is klt, we have
$$
\orbi{L}\left( \orbi{O}_{\orbi{X}},  \left|\Delta''\right|^{-2}\right)=\orbi{O}_{\orbi{X}}
$$
by~\cite[Def.~5.2.13]{sasakian}, compare also~\cite[Prop.~10.7]{KollarShaf}.

Now consider the orbifold universal cover $\pi \colon \widetilde{\orbi{X}} \to \orbi{X}$  of $\orbi{X}=(X,\Delta)$ and the $\Gamma$-reduction $\gamma\colon \orbi{X} \dasharrow \Gamma(\orbi{X})$ from Theorem~\ref{thm:G-red}. Let $F$ be a very general fiber of the restriction $\left. \gamma\right|_{X^0} \colon X^0 \to \Gamma(\orbi{X})^0$. The preimage $\pi^{-1}(F)$ is a disjoint countable union of copies of a finite cover of $F$. So the restriction of $\pi$ to a connected component $\widetilde{F}$ of this preimage is a finite cover $\left.\pi\right|_{\widetilde{F}}\colon \widetilde{F} \to F$, \'etale over $F \setminus \supp(\Delta)$. 

Denote by $\widetilde{L}$, $\widetilde{A}$, and $\widetilde{\Delta''}$  the pullbacks (as orbibundles) of $L$, $A$, and $\Delta''$ by $\pi$  respectively. In the same manner, denote the pullback of functions, orbimetrics, and orbiforms by a tilde. In particular, $(\widetilde{\orbi{X}},\widetilde{\omega}:=\pi^*\omega)$ is a complete K\"ahler orbifold.

If $\gamma(F) \subset U \subset \Gamma(\orbi{X})^0$ is a sufficiently small neighbourhood of $\gamma(F)$ biholomorphic to the unit ball $B_1(0) \subseteq \CC^n$, then the connected component $\widetilde{U}$ of $\pi^{-1} \circ \gamma^{-1}(U)$ containing $\widetilde{F}$ is a relatively compact open neighbourhood of $\widetilde{F}$ biholomorphic to $\widetilde{F} \times U$.

Since $E$ is effective, there is a nonzero section $\sigma \in H^0(X,E) \cong H^0(X,K_{\orbi{X}} \otimes L)$. We can pull back $\left. \sigma \right|_{\gamma^{-1}(U)}$ via $\left.\pi\right|_{\widetilde{F}}$ to get a section $\widetilde{\sigma} \in H^0(\widetilde{U}, K_{\widetilde{X}}\times \widetilde{L})$. Let $\rho:U \to [0,1]$ be a smooth cutoff function with $\rho \equiv 1$ on a neighbourhood of $\gamma(F)$. Then $(\rho \circ \gamma \circ \pi) \cdot \widetilde{\sigma}$ is a smooth $\widetilde{L}$-valued $(n,0)$-form on $\widetilde{\orbi{X}}$.

Take a positive integer $N>a$. By Lemma~\ref{le:thirdle}, we have $k \in \ZZ_{>0}$ and multivalued $L^2$-sections $s=\{s_i\}_{i=1,\ldots,k}$ of $\widetilde{A}^{\otimes a/N}$, such that their set of zeros $\widetilde{U} \cap (s)_0$ is compact and $\orbi{I}(s)$ is contained in the ideal sheaf $\orbi{I}_{\widetilde{F}}$. By shrinking $U$ if necessary, we can assume that $\widetilde{U} \cap (s)_0 =\widetilde{F}$.

Now we want to define a singular Hermitian metric on  $\widetilde{L}=\widetilde{A} \otimes \widetilde{\Delta''}$. Recall that we have a Hermitian metric $h_{m\Delta''}$ of the line bundle $\Delta''^{\otimes m}$. Define
$$
H_s:=  \widetilde{h_{A^{\otimes a}}}^{\otimes 1/a-1/N} \times  \frac{ \widetilde{h_{A^{\otimes a}}}^{\otimes 1/N}}{\left|s\right|^2} \times \frac{ \widetilde{h_{m\Delta''}}^{\otimes 1/m}}{\widetilde{\left|\Delta''\right|}^2}.
$$
This is a singular Hermitian metric of $\widetilde{L}$. Since $N>a$, the curvature $i\Theta(\widetilde{L},H_s) \geq (1/a-1/N)\widetilde{\omega}$ is positive. By klt-ness of $(X,\Delta'')$, we have 
$$
\orbi{L}\left( \orbi{O}_{\widetilde{\orbi{X}}},  \widetilde{\left|\Delta''\right|}^{-2}\right)=\orbi{O}_{\widetilde{\orbi{X}}},
$$
so $\mathcal{I}(H_s) =\mathcal{I}(s)$ as in the proof of Proposition~\ref{prop:compsubs}.

Now consider the $(n,0)$-form $(\rho \circ \gamma \circ \pi) \cdot \widetilde{\sigma}$ from above.
The $(n,1)$-form $\tau:=\delbar (\rho \circ \gamma \circ \pi) \cdot \widetilde{\sigma} = \widetilde{\sigma} \delbar (\rho \circ \gamma \circ \pi)$ is $\delbar$-closed and square-integrable with respect to $H_s$ and $\widetilde{\omega}$, because its support lies in the relatively compact $\widetilde{U} \setminus \widetilde{F}$ and the poles of $H_s$ lie in $\widetilde{F}$. 

By Proposition~\ref{prop:delbarest}, there is a $\widetilde{L}$-valued $(n,0)$-form $\upsilon$ on $\widetilde{\orbi{X}}$, with $\delbar \upsilon = \tau$, again square-integrable with respect to $H_s$ and $\widetilde{\omega}$. Now set $\nu:=(\rho \circ \gamma \circ \pi) \cdot \widetilde{\sigma} - \upsilon$. Applying $\delbar$, we see that $\nu$ is holomorphic, and since $\upsilon$ is integrable with respect to $H_s$, we have $\left.\upsilon\right|_{\widetilde{F}} \equiv 0$ and thus $\left.\nu \right|_{\widetilde{F}}$ is not trivial. 

As we know from the proof of Proposition~\ref{prop:compsubs}, $\left|s\right|^2$ is bounded. So there is a positive constant $c$, such that
$$
\int_{\widetilde{\orbi{X}}} \left|\upsilon\right|_{\widetilde{h_{A^{\otimes a}}}}^2 \dif V_{\widetilde{\omega}} \leq c \int_{\widetilde{\orbi{X}}} \left|\upsilon\right|_{H_s}^2 \dif V_{\widetilde{\omega}}.
$$
Moreover, since $(\rho \circ \gamma \circ \pi) \cdot \widetilde{\sigma}$ is supported on $\widetilde{U}$, it is square-integrable with respect to $\widetilde{h_{A^{\otimes a}}}$ and $\widetilde{\omega}$ as well. So $\nu$ is a nontrivial section of $H^0_{(2)}(\widetilde{\orbi{X}},K_{\widetilde{\orbi{X}}} \otimes \widetilde{L})$.
\end{proof}

\begin{proposition}
\label{prop:fundfin}
Let $(Y,D'+D'')$ be a weakly Fano variety, $f:X \to Y$ a log resolution with exceptional prime divisors $E_i$, and $m_i \in \ZZ_{>0}$ arbitrary. Then the smooth geometric orbifold $\orbi{X}:=(X,\Delta=f_*^{-1} D' + \sum (1-1/m_i)E_i)$ has \emph{finite orbifold fundamental group}
$$
\left|\pi_1(\orbi{X})\right|<\infty.
$$
\end{proposition}

\begin{proof}
Define the divisor $L$ as in Proposition~\ref{prop:nontrivsec}. Then by Proposition~\ref{prop:nontrivsec}, there is a nontrivial section $\nu \in H^0_{(2)}(\widetilde{\orbi{X}},K_{\widetilde{\orbi{X}}} \otimes \widetilde{L})$. 

For any $k \in \ZZ_{>0}$, the power $\nu^{\otimes 2k}$ is  a global  $L^1$-section of $(K_{\widetilde{\orbi{X}}} \otimes \widetilde{L})^{\otimes 2k}$. This is due to the fact that $\orbi{X}=\widetilde{\orbi{X}}/\pi_1(\orbi{X})$ is compact, see~\cite[p.~286]{GromovKaehler}. Here it is not necessary that $\pi_1(\orbi{X})$ acts freely on the complex analytic space $\widetilde{X}$. The Poincar\'e series 
$$
P(\nu^{\otimes 2k}):=\sum_{g \in \pi_1(\orbi{X})} g^*\nu^{\otimes 2k}
$$
converges and defines a holomorphic $\pi_1(\orbi{X})$-invariant section of $(K_{\widetilde{\orbi{X}}} \otimes \widetilde{L})^{\otimes 2k}$. 

Now consider products $\bigotimes_{k_i} P(\nu^{\otimes 2k_i})$ of these sections. 
Then~\cite[Prop.~3.2.A]{GromovKaehler} says that \emph{if $\pi_1(\orbi{X})$ is infinite}  there exists at least one $\kappa$ and at least two partitions $\kappa=\sum k_i$ and $\kappa=\sum k'_i$, such that 
$$
\frac{\bigotimes_{k_i} P(f^{\otimes 2k_i})}{\bigotimes_{k'_i} P(k^{\otimes 2k'_i})}
$$
is a nonconstant meromorphic $\pi_1(\orbi{X})$-invariant function on $\widetilde{\orbi{X}}$. 
Thus $\bigotimes_{k_i} P(f^{\otimes 2k_i})$ and $\bigotimes_{k'_i} P(k^{\otimes 2k'_i})$  define two linearly independent sections of $H^0(X,(K_{\orbi{X}} \otimes L)^{\otimes 2\kappa})$.

But on the other hand, since $E$ and thus also $2\kappa E$ is effective $f$-exceptional, we have $\dim H^0(X,\orbi{O}_X(2\kappa E))=1$.
 But we have seen that $E$ is linearly equivalent to $K_{\orbi{X}}+L$ (seen as a divisor). This is a contradiction, so $\pi_1(\orbi{X})$ is finite.
 
 \end{proof}
 
\subsection*{Proof of Theorem~\ref{thm:loctoglob}}
Now by the induction hypothesis, for an $n$-dimensional weakly Fano pair $(Y,D'+D'')$, we can relate finiteness of $\pi_1(Y_\sm,D')$ to the finiteness of the fundamental group of a compact orbifold supported on a log resolution, which finishes our proof.

\begin{proof}[Proof of Theorem~\ref{thm:loctoglob}]
Let $(Y,D'+D'')$ be an $n$-dimensional weakly Fano pair and assume that $n$-dimensional klt singularities  have finite regional fundamental group. Consider a log resolution $f:X \to Y$ with exceptional divisor $E=\bigcup E_i$, where $E_i$, $i \in I$ are prime. Then
$$
\pi_1(Y_\sm,D') \cong \pi_1(X \setminus E,\left.f_*^{-1}D'\right|_{X \setminus E}).
$$
Now let $\gamma_i$ be a very small loop around a general point $e_i$ of $E_i$. Then $\gamma_i$ can be pushed forward to $Y_\sm$ and there it lies in the smooth locus of a very small neighbourhood of the image of $e_i$, which is a klt singularity $y_i$. So by the induction hypothesis, $f_* \gamma_i$ has finite order $m_i$ in $\pi_1^{\reg}(Y,y_i)$. Therefore, it has finite order in $Y_{\sm}\cong X \setminus E$. 
Thus $\langle \gamma_i^{m_i}, i \in I \rangle$ is trivial and by Definition~\ref{def:orbifund} of the orbifold fundamental group 
\begin{align*}
\pi_1(X \setminus E,\left.f_*^{-1}D'\right|_{X \setminus E}) &=
\pi_1(X \setminus E,\left.f_*^{-1}D'\right|_{X \setminus E}) / \langle \gamma_i^{m_i}, i \in I \rangle
\\
&=
\pi_1\left(X,f_*^{-1}D' +\sum \left(1-\frac{1}{m_i}E_i\right)\right)
\end{align*}
But the latter is finite by Proposition~\ref{prop:fundfin}. Thus $\pi_1(Y_\sm,D')$ is finite as well and we are done.
\end{proof}

\part{Global to local}
\label{part:globtoloc}

In order to complete the induction, we have to show in this part that if $(n-1)$-dimensional weakly Fano pairs $(Y,D'+D'')$ have finite orbifold fundamental group $\pi_1(Y_{\sm},\left.D'\right|_{Y_{\sm}})$, then $n$-dimensional klt singularities have finite \emph{regional} fundamental group. 
We do this by modifying an argument of~\cite{TianXu}. First let us briefly recall the notions related to Whitney stratifications and their systems of tubular neighbourhoods. 

\section{Whitney stratifications}
\label{sec:Whit}

We refer to~\cite[Sec.~2]{Goresky} for the following definitions.
Let $X$ be a complex analytic space of dimension $n$ embedded in a smooth complex manifold $M$. In our context, we can always assume $M \cong \PP_m(\CC)$ for some $m\geq n$.

To a submanifold $N$ of $M$, we associate a \emph{tubular neighbourhood} $T_N$ in the following way: choose a Riemannian  metric $h$ on the normal bundle $E \to N$ and fix  $\delta \in \RR_{>0}$. Then $T_N$ is the image of a smooth embedding $\phi \colon E_{\delta} \to M$, where $E_{\delta}:=\{ v \in E;~\left|v\right|_h<\delta\}$ and $\phi$ takes the zero section of $E$ identically to $N$ . For $0<\varepsilon<\delta$, we define $T_N(\varepsilon):=\phi(\{ v \in E;~\left|v\right|<\varepsilon\})$ and its boundary $S_N(\varepsilon):=\phi(\{ v \in E;~\left|v\right|_h=\varepsilon\})$. We write $S_N:=S_N(\delta)$.  We have a \emph{tubular distance function} $\rho_N(x):=\left|\phi^{-1}(x)\right|_h$ and a \emph{projection} $\pi_N(x):=\phi \circ \pi \circ \phi^{-1}(x)$ both defined on $T_N$.

A \emph{Whitney stratification}  of $X \subseteq M$ is a filtration by closed subsets $X_0 \subset X_1 \subset \ldots \subset X_n=X$, such that the connected components of $X_i \setminus X_{i-1}$are locally closed $i$-dimensional submanifolds of $M$, the \emph{$i$-dimensional strata}. If $A$ and $B$ are strata with $A \cap \overline{B} \neq \emptyset$, then $A \subset \overline{B}$ and we write $A < B$. Any Whitney stratification allows a system of compatible tubular neighbourhoods of the strata, so called \emph{control data}. For two strata $A < B$, the tubular distance functions and projections from above have to satisfy $\pi_A \circ \pi_B = \pi_A$ and $\rho_A \circ \pi_B = \rho_A$. Moreover, for some $0<\varepsilon$,  the boundaries $S_N(\varepsilon)$ have to satisfy certain transversality properties. Namely if $A_1,\ldots,A_\mu$ and $B_1,\ldots,B_\eta$ are two disjoint collections of strata, then $S_{A_1}(\varepsilon) \cap \ldots \cap S_{A_\mu}(\varepsilon)$
and $S_{B_1}(\varepsilon) \cap \ldots \cap S_{B_\eta}(\varepsilon)$ are transversal and they are also transversal to any other stratum $C$. 

One can check these properties easily in the pictures of the next section.

\section{The work of Tian and Xu}
\label{sec:WorkTianXu}

It was shown in Lemmata 3.1 and 3.2 of~\cite{TianXu}, that if $(n-1)$-dimensional weakly Fano pairs $(Y,D'+D'')$ have finite orbifold fundamental group $\pi_1(Y_{\sm},\left.D'\right|_{Y_{\sm}})$, then $n$-dimensional klt singularities have finite \emph{local} fundamental group.

\subsection*{Tian and Xu's Lemma 3.4}
In~\cite[Le.~3.4]{TianXu}, finiteness of the  \emph{regional} fundamental group of an $n$-dimensional klt singularity $(X,x)$ is deduced from finiteness of the \emph{local} fundamental group of $k\leq n$-dimensional klt singularities. Unfortunately, there is a gap in the proof, as described in the following.

The proof uses a Whitney stratification of a neighbourhood of the singularity, together with a system of tubular neighbourhoods of the strata. Such a tubular neighbourhood minus the stratum itself is a fiber bundle over the stratum and the fiber over a point has finite fundamental group by assumption, since it is (homeomorphic to) a slice through a pointed neighbourhood of the point, which is klt.

Then the Seifert-van Kampen theorem is invoked to merge all these tubular neighbourhoods together to a neighbourhood of $x$ with the whole singular locus removed. The way the tubular neighbourhoods fit together is depicted below.

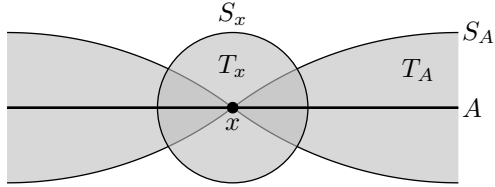
\begin{center}
\begin{tikzpicture}
\draw[line width=0.5pt, draw=black, domain=0:3, smooth, fill=lightgray,fill opacity=0.6] 
plot (3-\x,{-4+sqrt(-(3-\x)*(3-\x)+6*(3-\x)+16)}) 
--plot (\x,{4-sqrt(-\x*\x+6*\x+16)}) 
plot (-3+\x,{-4+sqrt(-(3-\x)*(3-\x)+6*(3-\x)+16)}) 
--plot (-\x,{4-sqrt(-\x*\x+6*\x+16)});

\draw[line width=0.5pt, draw=black, fill=lightgray,fill opacity=0.6] (0,0) circle (1cm);
\draw[black, fill=black] (0,0) circle [radius=2pt];
\draw[line width=1pt, draw=black] (-3,0)--(3,0);
\node[label=below:$x$] (A) at (0,0.1) {};
\node[label=right:$A$] (B) at (2.8,0) {};
\node[label=right:$T_A$] (B) at (2,0.5) {};
\node[label=right:$S_A$] (B) at (2.8,1) {};
\node[label=below:$T_x$] (A) at (0,0.95) {};
\node[label=above:$S_x$] (A) at (0,0.85) {};
\end{tikzpicture}
\end{center}

Here the singular locus $X_\sing=A \cup \{x\}$ has the zero-dimensional stratum $\{x\}$ and the one-dimensional stratum $A$, together with their tubular neighbourhoods $T_x$ and $T_A$. Note that their boundaries $S_x$ and $S_A$ are transversal. The regional fundamental group of $x$ is nothing else than the fundamental group of $U_x:=T_x \setminus X_\sing$. We have $T_x^0= U_x \cup (T_A \cap T_x)$. The intersection of $U_x$ and $T_A \cap T_x$ is just $T_A^0 \cap T_x$.
Thus we have canonical group homomorphisms $h_1 \colon T_A^0 \cap T_x \to U_x$ and $h_2 \colon T_A^0 \cap T_x \to T_A \cap T_x$.
 But since $T_A^0 \cap T_x$ and  $T_A \cap T_x$ both are  fiber bundles over $A \cap T_x$ with the fiber having finite fundamental group due to the assumption and trivial fundamental group, respectively, $h_2$ has finite kernel. 
 
Now the Seifert-van Kampen theorem says that $\pi_1(T_x^0)$ - which is finite by assumption - is the quotient of the free product $\pi_1(U_x) * \pi_1(T_A \cap T_x)$ by the normal subgroup $N$ generated by all elements $h_1(g)h_2(g)^{-1}$, where $g \in T_A^0 \cap T_x$.
The intersection of $N$ with $\pi_1(U_x) \subseteq \pi_1(U_x) * \pi_1(T_A \cap T_x)$ is nothing but the image under $h_1$ of the kernel of $h_2$.
 Now in the proof of~\cite[Le.~3.4]{TianXu}, it is argued that since $h_2$ has finite kernel, $\pi_1(U_x)$ can not be infinite. This is not necessarily true, since $h_1$ does not have to be surjective, that is, $h_1(\ker h_2)$ does not have to be a normal subgroup, and its normal closure can be infinite.  
 
 In fact, it is not hard to see that deducing finiteness of the regional fundamental group from finiteness of the local fundamental group is equally hard as deducing finiteness of $\pi_1(Y_{\sm},D')$ for weakly Fano pairs $(Y,D'+D'')$ from it. This is since $Y$ has a Whitney stratification as well and since we know that $Y$ is simply connected, the above arguments could be applied in the exact same manner.
 
 But as Tian and Xu suggested, one could try to modify Lemma 3.1 of~\cite{TianXu} to directly prove finiteness of the regional fundamental group. So let us have a close look at this lemma.

 \subsection*{Tian and Xu's Lemma 3.1}
 
 In~\cite[Sec.~3.2]{TianXu}, a klt singularity $x \in (X,\Delta)$ is blown up one time via $f:Y\to X$ such that the only exceptional prime divisor $E \subseteq Y$ (the so called Koll\'ar component) admits a divisor $\Delta_E$, such that $(E,\Delta_E)$ is weakly Fano~\cite[Rem.~1]{Xu}. Here $\Delta_E$ is the \emph{different} satisfying
 $$
 K_E + \Delta_E = \left.(K_Y + E + f_*^{-1}\Delta)\right|_E.
 $$ 
 On the other hand, we can write $\Delta_E=D'+D''$, where $D'$ is the 'different of zero' as defined in~\cite[Prop.-Def.~16.5]{Corti} by
 $$
 K_E + D' = \left.(K_Y + E )\right|_E.
 $$
 Then~\cite[Prop.~16.6]{Corti} says that on the one hand, $D'$ is of the form $\sum (1-1/m_i)D_i'$, and on the other hand, the $D_i'$ are \emph{singular strata in $Y$ of maximal dimension}, in particular, locally analytically at a general point of $D_i'$, $Y$ looks like $Z \times \CC^{n-2}$, where $Z$ is a two-dimensional cyclic quotient singularity $\CC^2/(\ZZ/m_i\ZZ)$. 
 
 It is shown in~\cite[Le.~3.1]{TianXu}, that the fundamental group of a certain open subset $V^0$ of a neighbourhood $U$ of $E$ surjects to the fundamental group of $U^0:=U \setminus E$, which is nothing else but the local fundamental group of $x$.
 
Then in~\cite[Le.~3.2]{TianXu}, it is shown that from finiteness of $\pi_1(E_{\sm},D')$ follows finiteness of $\pi_1(V^0)$.
 
 \begin{remark}
 It is not explicitly mentioned in~\cite{TX} that $(E,D'+D'')$ is log Fano but one has to consider the orbifold fundamental group $\pi_1(E_{\sm},D')$ - where $D'$ and $D''$ only coincide when the boundary $\Delta$ on $X$ is trivial. But by~\cite[Prop.~16.6]{Corti} this does not matter. This can be seen directly by observing that klt-ness of general points of singular strata of codimension two  does not depend on the boundary, since two-dimensional klt singularities are log-terminal.
 \end{remark}  

 Now if we could show that  $\pi_1(V^0)$ even surjects to $\pi_1(U_{\sm} \setminus E)$ - which is nothing but the regional fundamental group of $x$, we would be done. So how does the proof of~\cite[Le.~3.1]{TianXu} look like and how can it be modified?

After the blowup $f: Y \to X$, extracting the Koll\'ar component $E= f^{-1}(x)$, a Whitney stratification of $E$ is chosen, with biggest stratum $E_0 := E_{\sm}$. Choose $\varepsilon>0$. From the tubular neighbourhoods of the strata, a neighbourhood $U(\varepsilon)$ of $E$ in $Y$ is defined as follows, after~\cite[Def.~7.1]{Goresky}:
$$
U(\varepsilon):= \bigcup_{S \subseteq E \mathrm{~stratum}} T_S(\varepsilon).
$$
Note that strictly speaking, we have to embed $Y$ in a smooth manifold $M$ and the tubular neighbourhoods are neighbourhoods in $M$, not in $Y$. So as Goresky points out, the closure $\overline{U(\varepsilon)}$ in $M$ is a \emph{manifold with corners}, the corners being the intersections $S_{A_1}(\varepsilon) \cup \ldots S_{A_k}(\varepsilon)$, where $A_1 < \ldots < A_k$ are incident strata.
Nevertheless, we will denote the intersections of all these objects with $Y$ in the same way.
 We can draw a similar picture as before to depict the situation.
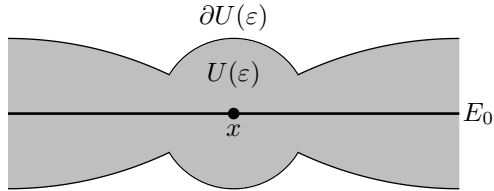
\begin{center}
\begin{tikzpicture}

\draw[line width=0.5pt, draw=black, fill=lightgray,fill opacity=1] (0,0) circle (1cm);

\draw[line width=0.5pt, draw=black, domain=0:3, smooth, fill=lightgray,fill opacity=1] 
plot (3-\x,{-4+sqrt(-(3-\x)*(3-\x)+6*(3-\x)+16)}) 
--plot (\x,{4-sqrt(-\x*\x+6*\x+16)}) 
plot (-3+\x,{-4+sqrt(-(3-\x)*(3-\x)+6*(3-\x)+16)}) 
--plot (-\x,{4-sqrt(-\x*\x+6*\x+16)});

\fill[fill=lightgray,fill opacity=1] (0,0) circle (28.2027559pt);
\draw[black, fill=black] (0,0) circle [radius=2pt];
\draw[line width=1pt, draw=black] (-3,0)--(3,0);
\node[label=below:$x$] (A) at (0,0.1) {};
\node[label=right:$E_0$] (B) at (2.8,0) {};
\node[label=below:$U(\varepsilon)$] (A) at (0,0.95) {};
\node[label=above:$\del U(\varepsilon)$] (A) at (0,0.85) {};
\end{tikzpicture}
\end{center}
Here $x \in E$ is the only stratum apart from $E_0$. The closure $\overline{U(\varepsilon)}$ has boundary $(S_{E_0}(\varepsilon) \cup S_x(\varepsilon)) \setminus (T_{E_0}(\varepsilon) \cup T_x(\varepsilon))$ and corners $S_{E_0}(\varepsilon) \cap S_x(\varepsilon)$. Now following~\cite[Sec.~7]{Goresky}, we construct a deformation retraction $\psi:U(\varepsilon) \to E$ as follows. 

First for every stratum consider a retraction $r_A \colon T_A(2\varepsilon) \setminus A \to S_A(2\varepsilon)$, such that the following hold whenever $A < B$ are incident strata:
$$
r_A \circ r_B = r_B \circ r_A,
\quad
\rho_A \circ r_B = \rho_A,
\quad
\rho_B \circ r_A = \rho_B,
\quad
\pi_A \circ r_A = \pi_A,
\quad
\pi_B \circ r_B = \pi_B.
$$
These retractions have been constructed in~\cite[Sec.~2]{GorFamLin} under the name \emph{families of lines}. From these one can define homeomorphisms $h_A \colon T_A(2\varepsilon) \setminus A \to S_A(2\varepsilon) \times (0,2\varepsilon)$, where $h_A(p)=(r_A(p),\rho_A(p))$. Now fix a smooth nondecreasing function $q$ with $q(t)=0$ for $t \leq \varepsilon$, $q(t)>0$ for $t > \varepsilon$  and $q(t)=t$ for $t \geq 2\varepsilon$. Now define
$$
H_A(p):= \left\lbrace
\begin{matrix}
p & \mathrm{if~} p \notin T_A(2\varepsilon) \setminus A \\
h_A^{-1}(r_A(p),q(\rho_A(p)) & \mathrm{if~} p \in T_A(2\varepsilon) \setminus A. \\
\end{matrix}
\right.
$$
Thus $H_A(\overline{T_A(\varepsilon)})=A$ and $H_A(T_A \setminus  \overline{T_A(\varepsilon)})=T_A$. Now define $\tilde{\psi}: U(2\varepsilon) \to U(2\varepsilon)$ by $\tilde{\psi}:= H_{A_1} \circ \ldots \circ H_{A_N}$, where $A_1,\ldots,A_N$ are the strata of $E$ in any order. Let $\psi:=\left.\tilde{\psi}\right|_{U(\varepsilon)}$. The restriction of $\psi$ to $E$ is homotopic to the identity~\cite[p.~220]{GM}. For each stratum $A$ and $\eta >0$ define the $\eta$-interior 
$$
A^\eta:= A \setminus \bigcup_{B < A} \overline{T_B(\eta)}
$$
as in~\cite[p.~180]{GM83}. With this definition, we see that $\psi(\pi_A^{-1}(A^\varepsilon))=A$ and $\psi(A \setminus A^\epsilon) \subseteq \bigcup_{B < A} B$. In our picture, $\psi$ collapses the (darkgray) $T_x(\varepsilon)$  to $x$ and the (lightgray) $\pi_{E_0}^{-1}(E_0^\varepsilon)$  to $E_0$:
 \begin{center}
\begin{tikzpicture}

\draw[line width=0.5pt, draw=black, fill=lightgray,fill opacity=1] (0,0) circle (1cm);

\draw[line width=0.5pt, draw=black, domain=0:3, smooth, fill=gray!20,fill opacity=1] 
plot (3-\x,{-4+sqrt(-(3-\x)*(3-\x)+6*(3-\x)+16)}) 
--plot (\x,{4-sqrt(-\x*\x+6*\x+16)}) 
plot (-3+\x,{-4+sqrt(-(3-\x)*(3-\x)+6*(3-\x)+16)}) 
--plot (-\x,{4-sqrt(-\x*\x+6*\x+16)});

\fill[fill=lightgray,fill opacity=1] (0,0) circle (28.2027559pt);
\draw[black, fill=black] (0,0) circle [radius=2pt];
\draw[line width=1pt, draw=black] (-3,0)--(3,0);
\draw[line width=1pt, draw=gray!80] (-3,0)--(-1,0) (1,0)--(3,0);
\node[label=below:$x$] (A) at (0,0.1) {};
\node[label=-160:$E_0^\varepsilon$] (B) at (3,0) {};
\node[label=right:$\pi_{E_0}^{-1}(E_0^\varepsilon)$] (B) at (1.35,0.5) {};
\node[label=below:$T_x(\varepsilon)$] (A) at (0,0.95) {};
\end{tikzpicture}
\end{center}

Now what is shown in~\cite[Le.~3.1]{TianXu}, is that the canonical group homomorphism
$
\pi_1(V^0) \to \pi_1(U(\varepsilon)\setminus E),
$
where $V^0:=\psi^{-1}(E_0) \setminus E=\pi_{E_0}^{-1}(E_0^\varepsilon) \setminus E$, is surjective. 

This is done by adding to $V^0$ closures in $U(\varepsilon)$ of the sets $V_A^0:=\psi^{-1}(A) \setminus E$ for all strata $A$, starting with those of highest dimension. Since all boundaries are collared, it is possible  to invoke the Seifert-van Kampen theorem, in order to show that $\pi_1(V^0) \to \pi_1(V^0 \cup \overline{V_A^0})$ is surjective, and so on. This is done by successive fiber bundle decompositions of $\overline{V_A^0}$. In order to really see what happens, we need a higher-dimensional picture with more strata.

 \begin{center}
 \tdplotsetmaincoords{70}{55}
 \tdplotsetrotatedcoords{0}{90}{0}
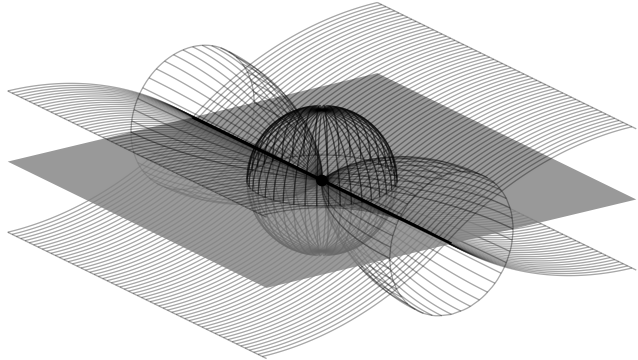
\begin{tikzpicture}[tdplot_main_coords]

\foreach \t in {-3,-2.9,...,3}
\draw[line width=0.5pt, draw=black, domain=0:3, smooth, opacity=\opac] 
plot (\t,3-\x,{4-sqrt(-(3-\x)*(3-\x)+6*(3-\x)+16)})
plot (\t,-3+\x,{4-sqrt(-(3-\x)*(3-\x)+6*(3-\x)+16)});;
\draw[line width=0.5pt, draw=black, domain=0:3, smooth, opacity=\opac] (-3,3,-1)--(3,3,-1) (-3,-3,-1)--(3,-3,-1);

\foreach \t in {-90,-81,...,90}
\draw[line width=0.5pt, draw=black, domain=0:3, smooth, opacity=\opac, tdplot_rotated_coords] 
plot (xyz cylindrical cs:radius={-4+sqrt(-(3-\x)*(3-\x)+6*(3-\x)+16)},angle=\t,z={3-\x});;
\foreach \t in {-90,-81,...,90}
\draw[line width=0.5pt, draw=black, domain=0:3, smooth, opacity=\opac, tdplot_rotated_coords] 
plot (xyz cylindrical cs:radius={-4+sqrt(-(3-\x)*(3-\x)+6*(3-\x)+16)},angle=\t,z={-3+\x});;
 \draw[line width=0.5pt, draw=black, domain=-90:90, smooth, opacity=\opac, tdplot_rotated_coords] 
plot (xyz cylindrical cs:radius={1},angle=\x,z={-3});

\foreach \t in {3,12,...,354}
\draw[line width=0.5pt, draw=black, domain=0:-180, smooth, opacity=\opac] 
plot (xyz spherical cs:radius=1,longitude={\t},latitude={\x});;

\fill[fill=gray,fill opacity=0.8] (-3,-3,0) -- (3,-3,0) -- (3,3,0) -- (-3,3,0) -- cycle;
\draw[black, fill=black] (0,0,0) circle [radius=2pt];
\draw[line width=1pt, draw=black] (-3,0,0)--(3,0,0);


\foreach \t in {-3,-2.9,...,3}
\draw[line width=0.5pt, draw=black, domain=0:3, smooth, opacity=\opac] 
plot (\t,3-\x,{-4+sqrt(-(3-\x)*(3-\x)+6*(3-\x)+16)})
plot (\t,-3+\x,{-4+sqrt(-(3-\x)*(3-\x)+6*(3-\x)+16)});;
\draw[line width=0.5pt, draw=black, domain=0:3, smooth, opacity=\opac] (-3,3,1)--(3,3,1) (-3,-3,1)--(3,-3,1);

\foreach \t in {-90,-99,...,-270}
\draw[line width=0.5pt, draw=black, domain=0:3, smooth, opacity=\opac, tdplot_rotated_coords] 
plot (xyz cylindrical cs:radius={-4+sqrt(-(3-\x)*(3-\x)+6*(3-\x)+16)},angle=\t,z={3-\x});;
 \draw[line width=0.5pt, draw=black, domain=0:360, smooth, opacity=\opac, tdplot_rotated_coords] 
plot (xyz cylindrical cs:radius={1},angle=\x,z={3});
\foreach \t in {-90,-99,...,-270}
\draw[line width=0.5pt, draw=black, domain=0:3, smooth, opacity=\opac, tdplot_rotated_coords] 
plot (xyz cylindrical cs:radius={-4+sqrt(-(3-\x)*(3-\x)+6*(3-\x)+16)},angle=\t,z={-3+\x});;
 \draw[line width=0.5pt, draw=black, domain=-90:-270, smooth, opacity=\opac, tdplot_rotated_coords] 
plot (xyz cylindrical cs:radius={1},angle=\x,z={-3});

\foreach \t in {3,12,...,354}
\draw[line width=0.5pt, draw=black, domain=0:180, smooth, opacity=\opac] 
plot (xyz spherical cs:radius=1,longitude={\t},latitude={\x});;
\draw[line width=0.5pt, draw=black, domain=0:360, smooth, opacity=\opac] 
plot (xyz spherical cs:radius=1,longitude={\x},latitude=0);
\end{tikzpicture}
\end{center}

Here, the horizontal plane depicts the divisor $E$, having three strata: the origin $o$, a one-dimensional stratum $A$, and the big open stratum $E_0$. It holds $\{o\}<A<E_0$. Also the ( boundaries of the) tubular neighbourhoods $T_N(\varepsilon)$ of  these strata are depicted, and their union is the open neighbourhood $U(\varepsilon)$ of $E$. Now we have 
$$
V:=\psi^{-1}(E_0)=\pi_{E_0}^{-1}(E_0^\varepsilon)=T_{E_0}(\varepsilon) \setminus (\overline{T_{A}(\varepsilon)} \cup \overline{T_{o}(\varepsilon)}),
$$
 which is depicted below, and in order  to get $V^0$ we have to subtract $E$. 

In a first step, the closure of $V_A^0:=\psi^{-1}(A)\setminus E$ has to be added to $V^0$. The Seifert-van Kampen theorem can be used to compute the fundamental group of the resulting space. Taking into account that all these spaces have collared boundaries, we can assume the intersection of $V^0$ and $V_A^0$ is $\del V_A^0 \cap T_{E_0}(\varepsilon)$, which is denoted $\orbi{L}_2$ in~\cite{TianXu}. 
Then if $\pi_1(\orbi{L}_2) \to \pi_1(V_A^0)$ is surjective, so is $\pi_1(V^0)\to \pi_1(\psi^{-1}(A \cup E_0)$.

But $\orbi{L}_2$ is a fiber bundle  over $A^\varepsilon$, with fiber $L_2$ homotopic to $\pi_A^{-1}(a) \cap  \del V \setminus E$ for some $a \in A^\varepsilon$, which is depicted in the cross-section through $a$ below.
 
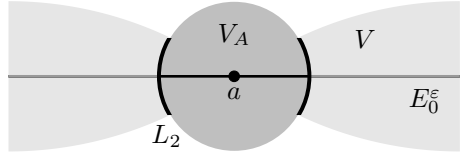
\begin{center}
\begin{tikzpicture}


\fill[ domain=0:3, smooth, fill=gray!20,fill opacity=1] 
plot (3-\x,{-4+sqrt(-(3-\x)*(3-\x)+6*(3-\x)+16)}) 
--plot (\x,{4-sqrt(-\x*\x+6*\x+16)}) 
plot (-3+\x,{-4+sqrt(-(3-\x)*(3-\x)+6*(3-\x)+16)}) 
--plot (-\x,{4-sqrt(-\x*\x+6*\x+16)});

\fill[fill=lightgray,fill opacity=1] (0,0) circle (28.2027559pt);
\draw[black, fill=black] (0,0) circle [radius=2pt];
\draw[line width=1pt, draw=black] (-3,0)--(3,0);
\draw[line width=1pt, draw=gray!80] (-3,0)--(-1,0) (1,0)--(3,0);
\node[label=below:$a$] (A) at (0,0.1) {};
\node[label=-160:$E_0^\varepsilon$] (B) at (3,0) {};
\node[label=right:$V$] (B) at (1.35,0.5) {};
\node[label=below:$V_A$] (A) at (0,0.95) {};

\node[label=below:$L_2$] (A) at (-0.9,-0.4) {};

\clip (-2,0.505) rectangle (2,-0.505);

\draw[line width=1.5pt, draw=black] (0,0) circle (1cm);
\end{tikzpicture}
\end{center}
On the other hand, also $V_A^0$ is a fiber bundle over $A^\varepsilon$, with fiber $L$ homotopic to $\pi_A^{-1}(a) \setminus E$ for some $a \in A^\varepsilon$. Thus if $\pi_1(L_2) \to \pi_1(L)$ is surjective, then so is $\pi_1(\orbi{L}_2) \to \pi_1(V_A^0)$.

Now $L_2$ and $L$ have a fiber bundle structure as well. There is a morphism $\varphi_A \colon V_A=\psi^{-1} \to \DD$, where $\DD=\{z \in \CC;~|z|<\varepsilon\}$, such that $Z_{A,0}:=\varphi^{-1}(0)=V_A \cap E$ and $\varphi_A$ is a topological fibration over $\DD^0:=\DD \setminus \{0\}$, see~\cite[p.~260]{TianXu}. Compare also the map $f$ in~\cite[Sec.~6.1]{GM83} and~\cite[Part~II,~Sec.~6.13.1]{GM}. In our picture, we see that approximately the fibers  $Z_{A,t}$ for $t \in \DD$ are horizontal sections of $V_A$.

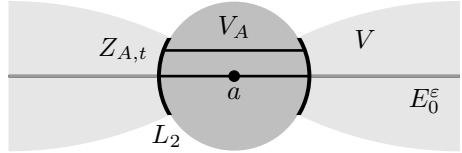
 \begin{center}
\begin{tikzpicture}


\fill[ domain=0:3, smooth, fill=gray!20,fill opacity=1] 
plot (3-\x,{-4+sqrt(-(3-\x)*(3-\x)+6*(3-\x)+16)}) 
--plot (\x,{4-sqrt(-\x*\x+6*\x+16)}) 
plot (-3+\x,{-4+sqrt(-(3-\x)*(3-\x)+6*(3-\x)+16)}) 
--plot (-\x,{4-sqrt(-\x*\x+6*\x+16)});

\fill[fill=lightgray,fill opacity=1] (0,0) circle (28.2027559pt);
\draw[black, fill=black] (0,0) circle [radius=2pt];
\draw[line width=1pt, draw=black] (-3,0)--(3,0);
\draw[line width=1pt, draw=gray!80] (-3,0)--(-1,0) (1,0)--(3,0);
\node[label=below:$a$] (A) at (0,0.1) {};
\node[label=-160:$E_0^\varepsilon$] (B) at (3,0) {};
\node[label=right:$V$] (B) at (1.35,0.5) {};
\node[label=below:$V_A$] (A) at (0,1.05) {};

\node[label=below:$L_2$] (A) at (-0.9,-0.4) {};

\clip (-2,0.505) rectangle (2,-0.505);

\draw[line width=1pt, draw=black] (-0.93969,0.34202)--(0.93969,0.34202);
\node[label=left:$Z_{A,t}$] (C) at (-0.9,0.34202) {};

\draw[line width=1.5pt, draw=black] (0,0) circle (1cm);
\end{tikzpicture}
\end{center}
So setting $Z_{a,t}:=Z_{A,t}\cap \pi_A^{-1}(a)$, we see that $L$ is a $Z_{a,t}$-bundle over $\DD^0$ and $L_2$ is a $\del Z_{a,t}$-bundle over $\DD^0$. So we have to show that $\pi_1(\del Z_{a,t}) \to \pi_1(Z_{a,t})$ is surjective. But $Z_{a,t}$ is homotopic to a collared affine analytic space of dimension $c$, where $c$ is the codimension of $A$ in $E$, see~\cite[Part~II,~Prop.~6.13.5]{GM}. Since $c\geq 2$, it follows that $\pi_0(\del Z_{a,t}) \to \pi_0(Z_{a,t})$ is an isomorphism and $\pi_1(\del Z_{a,t}) \to \pi_1(Z_{a,t})$ is surjective.

Repeating this procedure for all strata of $E$, Lemma 3.1 of~\cite{TianXu} is proven. 

\section{Finiteness of the regional fundamental group}
\label{sec:Le31mod}

In this section, we  prove Theorem~\ref{thm:globtoloc}, the global-to-local part of our induction, by modifying the proof of~\cite[Le.~3.1]{TianXu} appropriately.

\begin{proof}[Proof of Theorem~\ref{thm:globtoloc}]
As in Lemma 3.1 of~\cite{TianXu}, we start with an $n$-dimensional singularity $x \in X$ of a klt pair $(X,\Delta)$. We assume that the smooth locus of $(n-1)$-dimensional weakly Fano pairs has finite orbifold fundamental group. Let $f:Y \to X$ be a plt blowup extracting the Koll\'ar component $E=f^{-1}(x)$. Consider a Whitney stratification of $Y$, such that the biggest stratum is $Y_\sm$ and for $k\leq n-2$, the $k$-dimensional strata are the relative interiors - with respect to $Y_\sing$ - of the irreducible $k$-dimensional components of the singular locus $Y_\sing$. This induces a Whitney stratification of $E$ by cutting each stratum with $E$. Fix this stratification.

Let $0<\varepsilon<<1$ and  $U(\varepsilon)$ be a neighbourhood of $E$ as constructed in the previous section. Then $\pi_1^\reg(X,x)\cong \pi_1(U(\varepsilon)\setminus ( E \cup Y_\sing))$.
Again as in the previous section, construct the retraction $\psi \colon U(\varepsilon)\to E$.
Note that for any stratum of $Y_\sing$ there are two possibilities. Either it is of dimension $(n-2)$ \emph{and}  it is contained in $E$, and thus is of codimension one in $E$. Or it's intersection with $E$ is of codimension greater or equal to two in $E$. Now define
$$
E_0:=E \setminus \bigcup_{
\tiny{\begin{array}{c}
A \subseteq Y_\sing \mathrm{~stratum,}\\
\codim_E(A)\geq 2
\end{array}
}} A.
$$
Then if we choose $\varepsilon$ small enough, it is clear that $Y_\sing \cap U(\varepsilon)$ lies in $U(\varepsilon) \setminus \psi^{-1}(E_0)$. The situation is depicted below.

 \begin{center}
 \tdplotsetmaincoords{70}{55}
 \tdplotsetrotatedcoords{0}{90}{0}
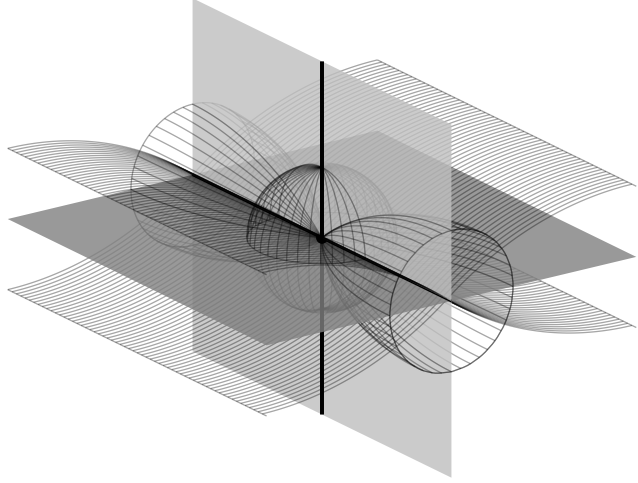
\begin{tikzpicture}[tdplot_main_coords]

\foreach \t in {-3,-2.9,...,3}
\draw[line width=0.5pt, draw=black, domain=0:3, smooth, opacity=\opac] 
plot (\t,3-\x,{4-sqrt(-(3-\x)*(3-\x)+6*(3-\x)+16)});;
\draw[line width=0.5pt, draw=black, domain=0:3, smooth, opacity=\opac] (-3,3,-1)--(3,3,-1);

\foreach \t in {0,9,...,90}
\draw[line width=0.5pt, draw=black, domain=0:3, smooth, opacity=\opac, tdplot_rotated_coords] 
plot (xyz cylindrical cs:radius={-4+sqrt(-(3-\x)*(3-\x)+6*(3-\x)+16)},angle=\t,z={3-\x});;
\foreach \t in {0,9,...,90}
\draw[line width=0.5pt, draw=black, domain=0:3, smooth, opacity=\opac, tdplot_rotated_coords] 
plot (xyz cylindrical cs:radius={-4+sqrt(-(3-\x)*(3-\x)+6*(3-\x)+16)},angle=\t,z={-3+\x});;
 \draw[line width=0.5pt, draw=black, domain=-90:90, smooth, opacity=\opac, tdplot_rotated_coords] 
plot (xyz cylindrical cs:radius={1},angle=\x,z={-3});

\foreach \t in {180,189,...,360}
\draw[line width=0.5pt, draw=black, domain=0:-90, smooth, opacity=\opac] 
plot (xyz spherical cs:radius=1,longitude={\t},latitude={\x});;

\fill[fill=lightgray,fill opacity=0.8] (-3,0,0) -- (-3,0,-2.5) -- (3,0,-2.5) -- (3,0,0) -- cycle;
\draw[line width=1.5pt, draw=black] (0,0,0)--(0,0,-2.5);

\foreach \t in {-3,-2.9,...,3}
\draw[line width=0.5pt, draw=black, domain=0:3, smooth, opacity=\opac] 
plot (\t,-3+\x,{4-sqrt(-(3-\x)*(3-\x)+6*(3-\x)+16)});;
\draw[line width=0.5pt, draw=black, domain=0:3, smooth, opacity=\opac] (-3,-3,-1)--(3,-3,-1);

\foreach \t in {-90,-81,...,0}
\draw[line width=0.5pt, draw=black, domain=0:3, smooth, opacity=\opac, tdplot_rotated_coords] 
plot (xyz cylindrical cs:radius={-4+sqrt(-(3-\x)*(3-\x)+6*(3-\x)+16)},angle=\t,z={3-\x});;
\foreach \t in {-90,-81,...,0}
\draw[line width=0.5pt, draw=black, domain=0:3, smooth, opacity=\opac, tdplot_rotated_coords] 
plot (xyz cylindrical cs:radius={-4+sqrt(-(3-\x)*(3-\x)+6*(3-\x)+16)},angle=\t,z={-3+\x});;
 \draw[line width=0.5pt, draw=black, domain=-90:90, smooth, opacity=\opac, tdplot_rotated_coords] 
plot (xyz cylindrical cs:radius={1},angle=\x,z={-3});

\foreach \t in {0,9,...,180}
\draw[line width=0.5pt, draw=black, domain=0:-90, smooth, opacity=\opac] 
plot (xyz spherical cs:radius=1,longitude={\t},latitude={\x});;


\fill[fill=gray,fill opacity=0.8] (-3,-3,0) -- (3,-3,0) -- (3,3,0) -- (-3,3,0) -- cycle;
\draw[black, fill=black] (0,0,0) circle [radius=2pt];
\draw[line width=1pt, draw=black] (-3,0,0)--(3,0,0);


\foreach \t in {-3,-2.9,...,3}
\draw[line width=0.5pt, draw=black, domain=0:3, smooth, opacity=\opac] 
plot (\t,3-\x,{-4+sqrt(-(3-\x)*(3-\x)+6*(3-\x)+16)});;
\draw[line width=0.5pt, draw=black, domain=0:3, smooth, opacity=\opac] (-3,3,1)--(3,3,1);

\foreach \t in {90,99,...,180}
\draw[line width=0.5pt, draw=black, domain=0:3, smooth, opacity=\opac, tdplot_rotated_coords] 
plot (xyz cylindrical cs:radius={-4+sqrt(-(3-\x)*(3-\x)+6*(3-\x)+16)},angle=\t,z={3-\x});;
 \draw[line width=0.5pt, draw=black, domain=0:360, smooth, opacity=\opac, tdplot_rotated_coords] 
plot (xyz cylindrical cs:radius={1},angle=\x,z={3});
\foreach \t in {90,99,...,180}
\draw[line width=0.5pt, draw=black, domain=0:3, smooth, opacity=\opac, tdplot_rotated_coords] 
plot (xyz cylindrical cs:radius={-4+sqrt(-(3-\x)*(3-\x)+6*(3-\x)+16)},angle=\t,z={-3+\x});;
 \draw[line width=0.5pt, draw=black, domain=-180:-270, smooth, opacity=\opac, tdplot_rotated_coords] 
plot (xyz cylindrical cs:radius={1},angle=\x,z={-3});

\foreach \t in {-90,-81,...,90}
\draw[line width=0.5pt, draw=black, domain=0:90, smooth, opacity=\opac] 
plot (xyz spherical cs:radius=1,longitude={\t},latitude={\x});;
\draw[line width=0.5pt, draw=black, domain=-90:90, smooth, opacity=\opac] 
plot (xyz spherical cs:radius=1,longitude={\x},latitude=0);


\fill[fill=lightgray,fill opacity=0.8] (-3,0,0) -- (-3,0,2.5) -- (3,0,2.5) -- (3,0,0) -- cycle;
\draw[line width=1.5pt, draw=black] (0,0,0)--(0,0,2.5);

\foreach \t in {-3,-2.9,...,3}
\draw[line width=0.5pt, draw=black, domain=0:3, smooth, opacity=\opac] 
plot (\t,-3+\x,{-4+sqrt(-(3-\x)*(3-\x)+6*(3-\x)+16)});;
\draw[line width=0.5pt, draw=black, domain=0:3, smooth, opacity=\opac]  (-3,-3,1)--(3,-3,1);

\foreach \t in {180,189,...,270}
\draw[line width=0.5pt, draw=black, domain=0:3, smooth, opacity=\opac, tdplot_rotated_coords] 
plot (xyz cylindrical cs:radius={-4+sqrt(-(3-\x)*(3-\x)+6*(3-\x)+16)},angle=\t,z={3-\x});;
 \draw[line width=0.5pt, draw=black, domain=0:360, smooth, opacity=\opac, tdplot_rotated_coords] 
plot (xyz cylindrical cs:radius={1},angle=\x,z={3});
\foreach \t in {180,189,...,270}
\draw[line width=0.5pt, draw=black, domain=0:3, smooth, opacity=\opac, tdplot_rotated_coords] 
plot (xyz cylindrical cs:radius={-4+sqrt(-(3-\x)*(3-\x)+6*(3-\x)+16)},angle=\t,z={-3+\x});;
 \draw[line width=0.5pt, draw=black, domain=-90:-180, smooth, opacity=\opac, tdplot_rotated_coords] 
plot (xyz cylindrical cs:radius={1},angle=\x,z={-3});

\foreach \t in {90,99,...,270}
\draw[line width=0.5pt, draw=black, domain=0:90, smooth, opacity=\opac] 
plot (xyz spherical cs:radius=1,longitude={\t},latitude={\x});;
\draw[line width=0.5pt, draw=black, domain=90:270, smooth, opacity=\opac] 
plot (xyz spherical cs:radius=1,longitude={\x},latitude=0);
\end{tikzpicture}
\end{center}

Here $Y_\sing$ has a $2$-dimensional stratum $Y_A$ that meets $E$ in the $1$-dimensional stratum $A$  and a $1$-dimensional stratum $Y_o$ that meets $E$ in the $0$-dimensional stratum $o$ ($A$ and $o$ as denoted in the last section).  Now as in the proof of~\cite[Le.~3.1]{TianXu}, start with $V_{E_0}^0:=\psi^{-1}(E_0) \setminus E$. But instead of adding (the closures of) $V_N^0:=\psi^{-1}(N) \setminus E$ to $V_{E_0}^0$ for all strata $N$ of $E \setminus E_0$, now we have to add $V_N^\sm:= V_N^0 \setminus Y_\sing$ in order to arrive at $U(\varepsilon)\setminus ( E \cup Y_\sing)$.

Now everything works the same way as in~\cite[Le.~3.1]{TianXu}, \emph{untill} we arrive at the $Z_{N,t}$ for some stratum $N$, compare the explanations in the previous section. Here, now we have to show that  $\pi_1(\del Z_{n,t} \setminus Y_\sing) \to \pi_1(Z_{n,t} \setminus Y_\sing)$ is surjective for an element $n$ of the $\varepsilon$-interior $N^\varepsilon$ in order to finish the proof. In our picture, for $N=A$, the situation looks like this.
 
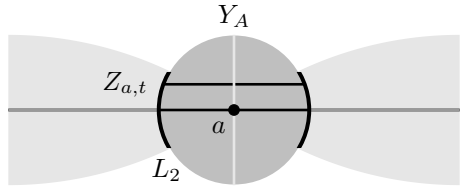
\begin{center}
\begin{tikzpicture}


\fill[ domain=0:3, smooth, fill=gray!20,fill opacity=1] 
plot (3-\x,{-4+sqrt(-(3-\x)*(3-\x)+6*(3-\x)+16)}) 
--plot (\x,{4-sqrt(-\x*\x+6*\x+16)}) 
plot (-3+\x,{-4+sqrt(-(3-\x)*(3-\x)+6*(3-\x)+16)}) 
--plot (-\x,{4-sqrt(-\x*\x+6*\x+16)});

\fill[fill=lightgray,fill opacity=1] (0,0) circle (28.2027559pt);

\draw[line width=1pt, draw=black] (-0.93969,0.34202)--(0.93969,0.34202);
\draw[line width=1pt, draw=gray!20] (0,-1)--(0,1);

\draw[black, fill=black] (0,0) circle [radius=2pt];
\draw[line width=1pt, draw=black] (-3,0)--(3,0);
\draw[line width=1pt, draw=gray!80] (-3,0)--(-1,0) (1,0)--(3,0);
\node[label=below:$a$] (A) at (-0.2,0.1) {};
\node[label=above:$Y_A$] (A) at (0,0.85) {};

\node[label=below:$L_2$] (A) at (-0.9,-0.4) {};

\clip (-2,0.505) rectangle (2,-0.505);

\node[label=left:$Z_{a,t}$] (C) at (-0.9,0.34202) {};

\draw[line width=1.5pt, draw=black] (0,0) circle (1cm);
\end{tikzpicture}
\end{center}
Note that in general, the singular locus $Y_\sing$ can have nontrivial intersection with $\del Z_{N,t}$. This is the case for example for $N:=\{o\}$, the zero-dimensional stratum, where $\del Z_{o,t}$ has nontrivial intersection with the $2$-dimensional stratum $Y_A$ of $Y_\sing$.

In~\cite[Proof~of~Le.~3.1]{TianXu}, it was argued that $Z_{a,t}$ is homeomorphic to an affine complex analytic space with collared boundary. This is due to~\cite[Part~II,~Prop.~6.13.5]{GM}. Looking into the proof therein, we see that this statement is obtained by using Thom's first isotopy lemma to show that $Z_{a,t}$ is homeomorphic to the intersection of $Z_{A,t}$ with smooth submanifolds of $M$ transversal to $A$ and a small euclidean ball around $a$. But this is an even stronger statement. It means that by this homeomorphy, we can assume that $\{x_a\}:=Z_{a,t} \cap Y_A$ is a klt singularity in some $c$-dimensional variety $Z$, and $Z_{a,t}$ in turn is the intersection of $Z$ with a small ball around $x_a$. Note that $x_a$ does not have to be isolated, since the singular locus $(Z_{a,t})_\sing=Z_{a,t} \cap Y_\sing$ in general is bigger. 
Nevertheless, we know that $\del Z_{a,t} \cap Y_\sing$ is nothing but the \emph{regional link} (i.e. $\Link(x_a) \cap Z_\sm$) of $x_a$ and thus $\pi_1(\del Z_{a,t} \cap Y_\sing)=\pi_1^\reg(Z,x_a)=\pi_1(Z_{a,t} \cap Y_\sing)$. 

By repeating this procedure for every stratum $N$ of $E$, we arrive at the surjection $ \pi_1(V_{E_0}^0) \to \pi_1(U(\varepsilon)\setminus (E \cup Y_\sing))$ as wanted. By Lemma~3.2 of~\cite{TianXu}, we know that $\pi_1(V_{E_0}^0)$ is finite due to the induction hypothesis, so $\pi_1^\reg(X,x)$ is finite. By~\cite[Le.~3.5]{TianXu}, also the regional orbifold fundamental group  $\pi_1^\reg(X,\Delta,x)$ is finite and the proof is finished.
\end{proof}

\printbibliography

\end{document}